\documentclass[11pt,leqno]{amsart}
\usepackage{mathrsfs,amssymb, color}
\usepackage{amscd}

\newenvironment {listeroman}
{\begin{enumerate}}
{\end{enumerate}}

\newtheorem{theorem}{Theorem}[section]
\newtheorem{lemma}[theorem]{Lemma}
\newtheorem{claim}[theorem]{Claim}

\newtheorem{corollary}[theorem]{Corollary}
\newtheorem{proposition}[theorem]{Proposition}
\newtheorem{definition}[theorem]{Definition}

\newtheorem{ssection}[theorem]{}

\let\<=\langle \let\>=\rangle
\DeclareMathOperator\Gr{\rm Gr}

\DeclareMathOperator\diam{\rm diam}

\numberwithin{equation}{section}

\let\nd=\noindent

\renewcommand\i{\ifmmode\mathchar"231 \else\char"10\fi}
\let\<=\langle  \let\>=\rangle
\newcommand\Q{{\mathbb Q}}
\newcommand\R{{\mathbb R}}
\newcommand \PP{{\mathbb P}}
\newcommand \I{{\mathbb I}}

\newcommand \bG{{\mathbf G}}
\newcommand \FF{{\mathcal F}}
\newcommand\KK{\mathcal K}

\newcommand \A {{\mathcal  A}}

\newcommand\II{{\mathcal I}}

\newcommand\impl{ \ \Longrightarrow \ }

\newcommand\ana{\boldsymbol{\Sigma}^1_1}
\newcommand\ca{\boldsymbol{\Pi}^1_1}
\newcommand\pca{\boldsymbol{\Sigma}^1_2}
\newcommand\borm[1]{\boldsymbol{\Pi}^0_{#1}}
\newcommand \bora[1]{\boldsymbol{\Sigma}^0_{#1}}
\newcommand \bord[1]{\boldsymbol{\Delta}^0_{#1}}
\newcommand\borel{\boldsymbol{\Delta}^1_1}

\newcommand \Gd{\boldsymbol{G}_\delta}
\newcommand \Fs{\boldsymbol{F}_\sigma}

\newcommand \abs[1]{\left\vert#1\right\vert}
\newcommand \adh[1]{\overline{#1}}
\newcommand\cat{{}^\frown}      

\newcommand\td[1]{\tilde #1}

\let\a=\alpha
\let\b=\beta
\let\wo=\omega
\let\s=\sigma
\let\d=\delta
\let\g=\gamma
\let\e=\varepsilon

\let\sk=\smallskip
\let\mns =\setminus

\let\ph=\varphi
\let\G=\Gamma

\catcode`\²=13 \let ²=\leq
\catcode`\³=13 \let ³=\geq

\let\0=\emptyset

\newcommand\red[1]{{\color{red}{#1}}}

\let\red=\relax
\begin{document}

\title[Closed zero-dimensional subspaces]{The descriptive complexity of the set\\ of all closed zero-dimensional 
subsets \\ of a Polish space}
\author{Gabriel Debs and Jean Saint Raymond}   
\begin{abstract} 
Given a   space $X$ we investigate the descriptive complexity  class $\G_X$  of the set  $\FF_0(X)$ of all  
 its closed zero-dimensional  subsets, viewed as a subset of the hyperspace $\FF(X)$ of all closed subsets of $X$. We prove that
  $\max \{ \G_X;  \ X \text{ analytic } \}=\pca $ and
  $\sup \{ \G_X; \ X \text{ Borel } \borm \xi\}   \supseteq \Game \bora \xi$  for any countable ordinal $\xi\geq1$. 
In particular we prove that there exists a one-dimensional Polish subpace of $2^\wo\times \R^2$   for which  $\FF_0(X)$ is not in the smallest  non trivial pointclass closed under complementation and the Souslin operation $\mathcal A\,$. \end{abstract}



\maketitle

\begin{section}{Introduction}

   If $X$ is a compact metrizable space and if we endow the hyperspace space $\FF(X)$ of all  closed subsets of $X$ with its canonical (compact metrizable) topology then (see Proposition  \ref{F0compact}) the  set $\FF_0(X)$ of all zero-dimensional  closed subspaces of $X$ is   a $\borm 2=\Gd$ subset of $\FF(X)$. 
When $X$ is   an arbitrary separable metrizable space there is no canonical topology on $\FF(X)$.
Nevertheless in this general context   $\FF(X)$   is still naturally endowed with a  canonical Borel structure (the Effros Borel structure) which   
 is isomorphic to the  Borel structure of an analytic (Polish) space if $X$ itself is analytic (Polish). 
 As we shall see  for any analytic space $X$  the set   $\FF_0(X)$ is $\pca$ (Proposition  \ref{F0=S12}) and   this bound is optimal (Theorem  \ref{F0=S12-complete}),   which answers Question 5.2 of  \cite{pol}.

  \begin{theorem}\label{pcacomplete} 
  There exists an  analytic space $X_\infty\subset 2^\wo\times \I$ for which the set   $\FF_0(X_\infty)$ is $\pca$-complete.
  \end{theorem}

 Here and in the sequel we denote by $\I$ the  interval $[0,1]$, and  by $\PP$ and $\Q$ the sets of   irrational and rational numbers respectively.  The definition of the notions of $\G$-complete and $\G$-hard set are recalled in Section \ref{DST}. We do  not know whether one can impose on the  space $X_\infty$ to be    Polish or even Borel.  However we shall prove  the following lower complexity bound:

\begin{theorem}\label{main} For any   countable  ordinal $\xi\geq 3$ 
there exists a   $\borm\xi$   set  $X_\xi\subset 2^\wo\times \I$ for which the set  $\FF_0(X_\xi)$ is   $\Game \bora \xi$-hard.
 \end{theorem}

For the definition of the  classes $\Game \G$ we refer the reader to  Section \ref{maxborel}. 
 {We  point out  that in \cite{pol} E. Pol and R. Pol   showed that   $\FF_0(\I\times\PP)$   is actually a $\ca\setminus \ana$ set  (note that $\ca=\Game\bora 1$ and $\ana=\Game\borm 1$).} We do know not whether the previous result holds for  $\xi=2$ but in this case we have the following:

\begin{theorem}\label{main2}     There exists a   Polish   space $Y \subset 2^\wo\times \I^2$ for which the set  $\FF_0(Y)$ is   $\Game \bora 2$-hard.
 \end{theorem}

We  point out  that the class  $\Game \bora 2$ is strictly larger than the more classical Selivanovski's class of  $C$-sets. In particular  the set $\FF_0(Y)$ in Theorem \ref{main2}  is not in  the  smallest   class containing all open sets and closed under complementation and  operation $\mathcal A\,$.  {This gives a partial answer to Question 9.12 of \cite{bcz} which asks about the complexity of $\FF_0(X)$  for an arbitrary  Polish   space $X$.

It is interesting to observe  that determining the exact complexity of  $\FF_0(X)$ for an explicit Polish space $X$ reveals to be a non trivial problem, as long as compactness is faraway. A typical example is the case of the space $\I^2\setminus\Q^2$ pointed out in \cite{pol} where the authors ask whether the set 
 $\FF_0(\I^2\setminus\Q^2)$  is $\ca$, a question still open.} However  it turns out that  the space  $Y$ of Theorem \ref{main2} can in fact   be embedded  as a  closed subset  of  $\I^5\setminus\Q^5$. It follows then:

\begin{theorem}      The  set   $\FF_0(\I^5\setminus\Q^5)$  is   $\Game \bora 2$-hard. \end{theorem}

The construction of the space    $Y$  in Theorem \ref{main2}  relies on a class of subsets of the plane   that we call {\it hypergraphs}. More precisely the space $Y$   will    be   a  $\Gd$   subset of $2^\wo\times H$ where $H$ is a $\Gd$  hypergraph. This contrasts with the fact  that  the set $\FF_0(H)$ for $H$ itself is a $\Gd$ set.  
The precise  definition of  an hypergraph  is a bit technical.
As a first approach one can say  that  an  hypergraph is  a subset $H$ of the plane satisfying
  \hbox{$G\subset H\subset \adh G $}
where $G$ is the graph of some maximal continuous  function  \hbox{$f: P \subset I \to \R $}  {whose domain} $P$ is a co-countable subset of some interval $I$ and   strongly  oscillating around any element  of  $I\setminus P$. 

\red{ We  only state here  two  simple properties of hypergraphs:

\begin{theorem}   An hypergraph   is  a   one-dimensional connected Polish space which  contains  no non-trivial continuum.
 \end{theorem} 
 
\begin{theorem} A closed subset of an hypergraph is zero-dimensional if and only if it is of empty interior.
 \end{theorem} 
}

 \end{section}

  \begin{section}{General notions and   notations}   

\begin{ssection}{\bf Descriptive classes}:\label{DST}  
{\rm
In all the paper we shall work in the frame of separable metrizable, mostly Polish,  spaces.
 We assume the reader to be familiar with basic notions and results from Descriptive Set  Theory, for example as presented in \cite{kc} or \cite{mo}.  

A set $A$ is always viewed as a subset  of some  Polish space $X$, and 
the  complexity  of  $A$ is then   relative to the given space $X$.
We recall the fundamental notion of {\it reduction}: given $A\subset X$ and  $B\subset Y$  the  pair $(A,X)$  is said to be   {\it reducible}  to the  pair $(B,Y)$  
if there exists a continuous mapping $\varphi: X\to Y$ such that $A=\varphi^{-1}(B)$.
A  {\it descriptive class}  (or more simply a {\it  class}) is  a class  of such pairs $(A,X)$  closed under
reducibility.

In fact we shall identify  a set $A$ with a   pair $(A, X)$ for  {\it some (any)} Polish space $X\supset A$ depending on the notion. In particular we shall view descriptive classes as classes of sets and when we   speak about the {\it complexity of  $A$}    we  mean  its   complexity relatively to {\it any}  Polish space $X$ in which $A$ embeds, that is the {\it absolute complexity} of $A$. But  when we   say that $A$  is   {\it reducible} to $B$  we mean that $(A,X)$  is      reducible   to    $(B,Y)$   for {\it some} Polish spaces $X\supset A$ and $Y\supset B$. 

 We shall consider  in the sequel     the classical  classes  that we list next, but also the less popular classes $\Game\G$ that we will present later on in   Section \ref{maxborel}:

\

 $\bora \xi$: the additive Baire class of rank $\xi$ 
 
 $\borm \xi$: the multiplicative Baire class of rank $\xi$
 
 \qquad (for $\xi=2$ we shall also use the classical  notation  $\Fs$  and $\Gd$) 

 $\ana$: the class of analytic sets.

 $\ca$: the class of co-analytic sets.

  $\boldsymbol{\Sigma}^1_2$:  the class  of projections of  co-analytic sets, also denoted by $PCA$.

 $\boldsymbol{\Pi}^1_2$:  the class  of complements of $PCA$ sets, also denoted by $CPCA$.

  $\boldsymbol{\Delta}^1_2 :=\boldsymbol{\Sigma}^1_2 \cap \boldsymbol{\Pi}^1_2$

\medskip

Given some class $\G$   a set $A$ (not necessarily in $\G$) is said to be 
 {\it $\G$-hard}  if   any set $A'\subset 2^\wo$  with $A'\in \G$ is reducible to $A$; if moreover $A$  is in $\G$ then $A$  is said to be  {\it $\G$-complete}.  If $A$ is    $\G$-hard  and   $A$ is reducible to   $B$ 
 (in particular if $A$ is a closed subset of $B$) then $B$ is $\G$-hard too.

 To any class $\G$ we associate its {\it dual class} $\check \G= \{A^c: \ A\in \G\}$. If $\G$  is not {\it self-dual} ($\G\not=\check \G$) then any   $\G$-complete set   is in $\G\setminus\check \G$
 
Computing  the exact complexity  of a given set $A$  is finding a class $\G$ such that $A$ is $\G$-complete. When this is not accessible one then looks for  a lower  (upper) bound of this complexity that is a class $\G$  such that 
$A$ is  $\G$-hard  ($A$ is in $\G$) 

\sk

For any mapping $f:X\to Y$ we denote by $\Gr(f)$ its graph.}  
\end{ssection}

\begin{ssection} {\bf Sections in product spaces:} \label{notation} 
 {\rm We shall very often work in product spaces $X\times Y$ and we need to fix a number of practical conventions. Given a  set $E\subset X\times Y$ and  an element  $a\in X$ (the first factor) we denote
 by
 $$E(a)=\{y\in Y: \ (a,y)\in E\}$$ 
 the section of $E$ at $a$  parallel to  the second  factor.   
 We also associate to $E$ the set 
  $$\td E= \{(x,(x,y))\in X\times (X \times Y): \ (x,y)\in E\}\subset X \times E$$ 
hence
$$\td E (a)=\{a\}\times E(a) \subset E. $$

\nd Observe that as a subset of $Y$  the set $E(a)$ might be as complicated as the set  $E$  while
 $\td E (a)$ is a closed subset of  $E$.  
}
\end{ssection} 

\begin{ssection}{\bf Admissible topologies on $\FF(X)$:}\label{F(X)}
{\rm 
For any  topological space $(X, \tau)$ we denote by $\FF(X)$ the space of all closed  
subsets of $X$.

We recall that the   Vietoris   topology on  the hyperspace  $\FF(X)$ is  the  topology $\td\tau$  generated by the sets of the form:
$$\FF  (V)=\{F\in \FF(X): \ F\subset V\} \quad \text{ and }  \quad  
\FF ^+(V)=\{F\in \FF(X): \ F\cap V \not=\0\}$$ 
 where $V$ is an arbitrary open subset of $X$.  It is well known that  if  $(X, \tau)$ is compact then 
$(\FF(X), \td\tau)$ is compact  too; and if moreover $X$ is metrizable then $\td\tau$   is the topology induced by the  Hausdorff metric associated to any compatible metric on  $X$.

For an arbitrary separable metrizable  space $(X, \tau)$ the topology $\td\tau$ is no more metrizable. However  one can embed  $X$ in a compact space $\hat X$ (not necessarily densely)     and then identify any set $F\in \FF(X)$ with   $\adh F$  (its  closure    in $\hat X$). The set $\FF(X)$ is then  identified  to  the set:
$$\FF(X, \hat X)=\{K\in \FF(\hat X): \ \adh{K\cap X}=K\} $$   
Observe that conversely one can recover $\FF(X)$ from $\FF(X, \hat X)$ since:
$$\FF(X)=\{K\cap X: \ K\in \FF(X,\hat X)\}. $$   
This identification   induces on $\FF(X)$ a topology $\td\tau_{\hat X}$ which depends on the specific compactification $\hat X$.

A topology on $ \FF(X)$ will   be said to be  {\it admissible} if it is of the form $\td\tau_{\hat X}$  for {\it some} compactification $\hat X$.   
We point out that any two  admissible topologies    on $\FF(X)$ are first Baire class isomorphic.
 In particular  if $\G$  is any descriptive class which is invariant  by first Baire class isomorphisms, and 
 $\td\tau_{\hat X_1}$ ,  $\td\tau_{\hat X_2}$ are  two compactifications  of $X$,  a set $S\subset \FF(X)$ is in $\G(\td\tau_{\hat X_1})$ iff  $S$ is in $\G(\td\tau_{\hat X_2})$. This  is in particular the case if $\G$ is a projective class or a Baire class of infinite rank.}  
 
\

{\it In all the sequel   $ \FF(X)$ will       implicitly be  endowed with an admissible topology}.

\nd {\rm We state next the main properties (see \cite{kc}  Section 12 C) that we will  use in the sequel: 
\medskip

{\it
{\bf a)} If $X$ is  analytic (Polish) then $\FF(X)$  is analytic (Polish).
\smallskip

{\bf  b)}  If $A$ is open (closed, compact) in $X$ then the set  $\{F\in  \FF(X):  F\cap A\not=\0\}$    

\quad is open (analytic, closed) in $\FF(X)$.
\smallskip
 
 {\bf  c)} The set   $\{(F,F')\in  \FF(X)\times\FF(X): \ F\subset F'\}$ is   Borel  in  $\FF(X)\times\FF(X)$.}

\medskip

We mention that one can find in the literature other natural Polish topologies on the hyperspace
of closed subsets of a Polish space $X$, such as the  {Wijsman}  topology. In most of these topologies 
the sets of the form $\FF^+(V)$ is open for $V$ open subset of $X$; it follows then  {that} the Borel structure
generated by such a topology coincides with the  Borel structure  generated by any admissible topology, that is the  Effros Borel structure. Also  as far as one is interested
in  descriptive classes of higher complexity than the Borel class, which is the case for our study, the precise choice of the topology is irrelevant.
 
}\end{ssection}

\medskip
\begin{ssection}{\bf Topological dimension:}
{\rm 
We recall the following  fundamental result  from Dimension Theory, that  we shall very often use in the sequel:}
\end{ssection}

\begin{theorem}\label{dim}   Suppose that $X=\bigcup_{n\geq 0} X_n$.  If   for all  $n > 0$, $X_n$  is  a closed subset  of $X$ and  $\dim(X_n)=0$  then $\dim(X)=1+\dim (X_0)$. 
  \end{theorem}

In particular if $X_0=\0$ one gets: 

\begin{theorem}\label{dim0}   If $X=\bigcup_{n\geq 0} X_n$ and for all  $n \geq 0$, $X_n$  is  a closed subset  of $X$ and  $\dim(X_n)=0$  then $\dim(X)=0$. 
  \end{theorem}  
\end{section}


\begin{section}{Computing the descriptive complexity of  $\FF_0(X)$}\label{F0(X)}

\begin{ssection}
{\bf Pseudo-graphs and canonical mappings:} {\rm  The graph   of a   mapping
  \hbox{$\Phi: X\to \FF(Y)$}  is of course a  subset  of $X\times \FF(Y)$
 but one can also consider the  graph of the canonical set-valued mapping (with possibly empty values) associated to $\Phi$  that is the set:  
  $$E= \{(x,y)\in X\times Y: \ y\in \Phi(x)\}.$$ 
that we shall call the {\it pseudo-graph} of $\Phi$.
Observe that if  $(U_n)$ is any  countable basis  of   topology in $Y$ and for all $n$,
$A_n=  \{x\in X: \  \Phi(x)\cap U_n\not=\0\}$
then  $E=\bigcap_n (X\times U_n^c)  \cup (A_n\times U_n)$.
It follows from \ref{F(X)}.a) that if  the mapping  $\Phi$ is   Borel (continuous)  then  $E$  is a Borel ($\Gd$) subset of   $X\times Y$ with closed sections. Conversely any subset $E$ of $X\times Y$ with closed sections is the pseudo-graph of a uniquely determined  mapping  $\Phi_E: X\to \FF(Y)$ that we call the {\it canonical mapping   associated to $E$}. But  if $E$ is Borel or even $\Gd$ the    mapping $\Phi_E$ is not Borel in general. However as we shall see in next result  a kind of converse is true if we work  in one additional dimension. 

We  recall that to any set $E\subset X\times Y $  we   associate  the closed subset of $X \times E$ defined by:
 $$\td E= \{(x,(x,y))\in X\times (X \times Y): \ (x,y)\in E\}\subset X \times E.$$ 
Since all   sections of $\td E$ are  closed subsets of $E$ one can consider then the  canonical 
mapping  $\Phi_{\td E}: X\to \FF(E)$. It turns out that the mapping  $\Phi_{\td E}$ has a better behavior than $\Phi_E$. Still if $E$ is Borel  or even $\Gd$
the mapping $\Phi_{\td E}$ is not Borel in general. However we have the following:
}\end{ssection}

 \begin{theorem}\label{hatE}
  Let $X$ and $Y$ be two separable metrizable spaces. Given any set  $E\subset X\times Y $ and any admissible topology on $\FF(E)$ there exists  a   subset  $\hat E$ of $X\times E$   satisfying:

a) $\tilde E \subset \hat E\subset X\times E.$
   
 b) $\hat E$ is closed in $X\times E$  and for all $x\in   X$,    $ \hat E (x)\setminus \td E(x) $ is  countable.  
 
 c) The canonical mapping  $\Phi_{\hat E}: X \to \FF(E)$ is of the first Baire  class and even  

\quad continuous  if $X$ is zero-dimensional.
\end{theorem}

\begin{proof}  Let $Z$ be a compactification of $E$ defining the admissible topology of $\FF(E)$. We fix   two compatible  metrics  $d_0$ on $X$ and $d_1$ on $Z$  and equip $X\times Z$ with the sum distance $d$:
$$
d\bigl( (x,z)\,, \, (x',z')\bigr)=d_0(x,x') + d_1(z,z').$$
We  fix a countable basis $(W_n)_{n\in\wo}$ of   the topology of $Z$
such that for all $n$, $W_n$ is nonempty and for all $\e>0$ the set $\{n\in \wo : \diam (W_n)>\e\}$  is finite. 
We then fix for all $n$ an element $e_n \in  W_n\cap E$, and  set $A_n=\{ x\in X: \ \td E(x)\cap W_n\not=\0\}$ and  $B_n=\{ x\in X: \ d_0(x, A_n)\leq  2^{-n} \}$.  Finally let:
$$
D=\bigcup_{n\in \wo} B_n\times \{e_n\} \quad \text{ and } \quad \hat E= \td E \cup D
$$ 
 So clause  a) is trivially  satisfied and for every $x\in X$, $D(x)= \{ e_n : B_n\ni x\}$ is countable.

 \medskip
 
  We  first prove  that $\hat E$ is   closed in  $X\times E$.   For this observe   that if $x\in B_n$ then by definition there exists some $x'\in A_n$ such that  $d_0(x, x')\leq  2^{-n}$, and since $x'\in A_n$ then  there exists some $z \in W_n$ such that $(x',z)\in \td E$   hence since $e_n\in W_n$
 we have
$$d\bigl( (x,e_n)\, , \, \td E\bigr) \leq  d\bigl( (x,e_n)\, ,\, (x',z)\bigr)= d_0(x, x')+ d_1(e_n, z) \leq  2^{-n} + \diam (W_n)$$
It follows that  if  a sequence $(x_j, e_{n_j})_j$ in $D$ with $x_j\in B_{n_j}$ converges to some element $(x,z)$ then:

-- either for infinitely many $j$'s,  $n_j=n$ is constant  hence $e_{n_j}=e_n$ and  $x_j\in B_n$, then $z=e_n$ and since $B_n$ is closed then $x\in B_n$; so
 $(x,z)\in D\subset \hat E$,

-- or $\lim_j n_j=\infty$   and in this case the inequality above shows that  $(x,z)\in \td E$.
 
 \sk
 
 So any accumulation point of $D$ lies in $\td E$, hence
  $\adh{D}^{X\times E}\subset D\cup \td E$ ;  and since  $\td E$ is a closed subset of $X\times E$ 
  then  $\hat E$ is a closed subset of $X\times E$.
 In the  same way   if  $F= \adh{E}^{X\times Z}$ then  
$\adh{D}^{X\times F} \subset \hat F:= \td F \cup D$ and $\hat F$ is a closed  subset of $X\times F$.


\smallskip

\begin{claim}: For all $x\in X$, $\hat F(x)=\adh{\hat E(x)}=\adh{D(x)}$.\end{claim}

\begin{proof} Since $\hat F$ is closed then for all $x\in X$, $\hat F(x)\supset\adh{\hat E(x)}$;
 and we  shall prove that if $W_n\cap  \hat F(x)\not=\0$ then $W_n\cap  D(x)\not=\0$,
 which proves that $\hat F(x)\subset\adh{D(x)}$. 

 If $W_n\cap D(x)\not=\0$ then we are done. Otherwise $W_n\cap  \td F(x)\not=\0$ and  we can pick some
$z\in W_n$ such that $(x,z)\in \td F=\adh{\td E}$. Then there exists some $(x',z')\in \td E$   such that $d_0(x,x')<2^{-n}$ and
$z'\in W_n$ hence $\td E(x')\cap W_n\not=\0$ and so $x'\in A_n$ and $x\in B_n$ and then $(x,e_n)\in \hat E$ so $e_n\in W_n\cap  D(x)$.
\end{proof}
 
  So clause b) is satisfied and we now  prove  clause c). 
 Since  $\FF(E)$ is endowed with the  admissible topology associated to the compact space $Z\supset E$
then by definition, the study of the mapping $\Phi_{\hat E}: X \to \FF(E)$ reduces  to the study of the mapping $\Psi: X\to \FF(Z)$ defined by 
$$\Psi (x)=  \adh{\Phi_{\hat E} (x)}=  \adh{\hat E(x)}=\hat F(x)$$
Since $\hat F$ is closed in $X\times Z$ and $Z$ is compact then for all open set $W$ in $Z$ 
the set $\{x\in X : \ \hat E(x) \subset W\}$ is open. Moreover it follows from the Claim that  for all $n$:
$$\{x\in X : \ \hat F(x) \cap W_n\not =\0\}=\bigcup\{ B_p:\  e_p\in W_n\}$$
which is an $\boldsymbol {F}_\s$ set, hence $\Psi$ is of first Baire class. 

Finally if $X$ is zero-dimensional we can choose for $d_0$   a  metric valued in $\{0\}\cup\{2^{-n} : \ n\geq 0\}$, for example by embedding $X$ in $2^\wo$ equipped with its standard ultrametric defined for $\a\not=\b$ by
$d_0(\a, \b)= 2^{-n(\a,\b)}$ with    $n(\a,\b)=\min\{ n: \ \a(n)\not=\b(n)\}$.
It follows  that the sets $B_n$  above are clopen  and $\Psi$ is then continuous. 
 \end{proof}

\begin{corollary}\label{E0} 
Let $X$ be a Polish zero-dimensional space, $Y$  a separable metrizable  space and   $E\subset X\times Y$. Then    the set $E_0=\{x\in X: \ \dim(E(x))=0\}$ is reducible to  $\FF_0(E)$.
\end{corollary}

\begin{proof} Let ${\hat E}\subset X\times E$ be the set given by Theorem \ref{hatE}. Observe that since $\td E$ is closed in $\hat E$
then it follows from clause b) and Theorem \ref{dim} that   
$$x\in E_0\iff   \dim(E(x))=0 \iff   \dim(\td E(x))=0\iff   \dim(\hat E(x))=0$$
 then by clause c)  the mapping $\Phi_{\hat E}$ realizes a reduction of  $E_0$ to $\FF_0(E)$ with all the desired 
 properties. 
\end{proof}

\end{section}
\begin{section}{\bf The $\s$-ideal structure of  $\FF_0(X)$} \label{s-ideal}

\begin{definition} \label{def-s-ideal}Given  a set $X$ and $\II\subset\A\subset\mathcal P(X)$ we shall say that 
$\II$ is  a $\s$-ideal of $\A$ if for any set $A\in \A $,  if 
$A\subset \bigcup_{n \geq 0}A_n$ with $A_n\in \II$ then $A \in \II$.
\end{definition}

So by Theorem \ref{dim0}  for any space $X$ the set  $\FF_0(X)$ is a  $\s$-ideal of $\FF(X)$.  

\medskip

 {For any space $X$ let $\KK(X)$ denote the space of all compact subsets of $X$ equipped with  the Vietoris topology.} Observe that if $X$ is a compact space,  a  $\s$-ideal  of $\KK(X)$ in the sense of Definition \ref{def-s-ideal} is    a  ``{\it \hbox{$\s$-ideal}  of compact sets}"  in the sense of \cite{klw}. We also  recall that  in this case ($X$compact) if $\mathcal C\subset \KK(X)$ is compact then the set $C=\bigcup \mathcal C\subset X$ is compact.

 More generally given a  set $\A\subset \FF(X)$ (endowed with some admissible topology) we shall say that $\A$ is {\it  closed under compact unions} if for any  compact  set  $\mathcal C\subset \A$ the  set  $C=\bigcup \mathcal C$ is a member of $\A$ (in particular  $C\in\FF(X)$).  

\begin{lemma}\label{compu}  For any admissible topology  $\FF(X)$ is closed under compact unions. 
\end{lemma}

\begin{proof}
Let $\hat X$ be a compactification   of $X$ defining  the given
admissible topology  on $\FF(X)$. We recall that $\FF(X)$ is  then 
homeomorphic  to the subset   
$\FF(X, \hat X)$ of  $\KK(\hat X)$ (see Section \ref{F(X)}). We now check that   
$\FF(X, \hat X)$  is  closed under compact unions:
 Suppose that $\mathcal C\subset \FF(X, \hat X)$ is    compact  and let $C=\bigcup \mathcal C$ then 
$$\adh{C\cap X}= \adh{\bigcup_{K\in\mathcal C} K\cap X}\supset\bigcup_{K\in\mathcal C} \adh{K\cap X}=
\bigcup_{K\in\mathcal C} K=C$$ 
So  $\adh{C\cap X}=C$ and  thus $C\in\FF(X, \hat X)$. 
 \end{proof} 

It turns out that most results of \cite{klw} concerning $\s$-ideals  of compact sets can be extended, with the same proofs, to   $\s$-ideals   of any   $\ana$  set $\A\subset \mathcal K(X)$    closed under compact unions. In particular  Theorem 7 and Theorem 11 of  \cite{klw} can be restated as follows:
 
 \begin{theorem} \label{KLW} {\sc (Kechris-Louveau-Woodin)}  Let $X$ be compact space and   $\II$  be a $\s$-ideal of $\A\subset\mathcal K(X)$.
If   $\A$ is $\ana$ and   is  closed under compact unions then:

a) If $\II$ is $\ana$ then $\II$ is a $\Gd$ subset of $\A$

b)  If $\II$ is $\ca$ then  $\II$ is either  a $\Gd$ subset of $\A$ or is $\ca$-complete. 
 \end{theorem}

\begin{corollary} \label{KLWI}  
  Let   $X$ be  an analytic space and   $\II$  be a $\s$-ideal of $\mathcal F(X)$.  

a) If $\II$ is $\ana$  for some admissible topology  then for any admissible topology  $\II$ is a $\Gd$ subset of $\FF(X)$ . 

b)  If $\II$ is $\ca$  for some admissible topology  then  for any admissible topology  $\II$ is either  a 
$\Gd$ subset of $\FF(X)$ or  $\ca$-complete.  
 \end{corollary} 

\begin{proof} As mentioned above any two  admissible topologies on $\FF(X)$ are first Baire class isomorphic, hence if $\II$ is $\ana$ for some admissible topology  then $\II$ is $\ana$ for any admissible topology. Corollary \ref{KLWI} follows then from    Lemma \ref{compu} and  Theorem~\ref{KLW}.
\end{proof}

   We recall that $\FF(X)$ is implicitly endowed with some admissible topology.
 
 \begin{corollary} \label{KLWF(X)}  
  Let   $X$ be  an analytic space.  

a) If $\FF_0(X)$ is $\ana$  then  $\FF_0(X)$ is a $\Gd$ subset of $\FF(X)$ . 

b)  If $\FF_0(X)$ is $\ca$   then  $\FF_0(X)$ is either  a 
$\Gd$ subset of $\FF(X)$ or  $\ca$-complete.  
 \end{corollary}

We emphasize  that  in  both statements    a) and b)  of Corollary \ref{KLWF(X)} the $\Gd$ condition is relative to $\FF(X)$ and  not absolute, unless the space $X$ itself, hence $\FF(X)$, is Polish.  In fact  if one wants to ensure that $\FF_0(X)$ is an (absolute) $\Gd$ then requiring $X$ to be Polish is not a restriction  since $X$  embeds as a closed subset of  $\FF_0(X)$, and in this context  we have the following complement to
Corollary \ref{KLWF(X)}. 

\begin{proposition} \label{Fs} If $X$ is  a Polish  space and $\FF_0(X)$  is an 
 $\Fs$   subset of  $\FF(X)$ then $X$ is zero-dimensional, hence $\FF_0(X)= \FF(X)$. 
 \end{proposition}

\begin{proof} Embed $\FF(X)$ in $\FF(\hat X)$  for some compactification  $\hat X$  of $X$ and let 
$\td\tau$ denote  the corresponding admissible topology on $\FF(X)$.   Let $U$ be the largest open subset of $X$ such that $\dim(U)=0$ (such a set  exists by Theorem \ref{dim0}) and suppose that the set  $Y=X\setminus U$ is nonempty.  

 Since $X$ is Polish then $Y$ as well as $\FF(Y)$ are  Polish too.    It follows from the hypothesis that  $\FF_0(Y)= \FF_0(X)\cap  \FF(Y )$ is an $\Fs$  subset  of $\FF(Y)$, hence by  Corollary \ref{KLWI} $\FF_0(Y)$ is a  also a $\Gd$   subset of   $ \FF(Y)$ which is  moreover dense since it contains all finite subsets. 
Hence   by Baire Theorem  $\FF_0(Y)$ has  a dense  interior in {$\FF(Y)$}. In particular  since $\0$ is an isolated element of $\FF(Y)$ the set 
$\FF_0^*(Y)=\FF_0(Y)\setminus \{\0\}$  is   of nonempty interior in  {$\FF(Y)$}. It follows then from the very definition of the topology $\td\tau$ that  there exists some finite family  $(V_j)_{0\leq j\leq k}$ of open subsets of $\hat X$  such that setting $V=\bigcup_{j=0}^k V_j$  then $V\cap Y\not=\0$ and
 $$ \mathcal V =\{ K\cap Y : K\in \KK( \hat Y)\ ,\ K \subset V \text{ and } \forall j\leq k, \  K\cap V_j\neq\0\}
$$
 is a nonempty subset  of $\FF_0^*(Y)$. We then fix for each $j$ an element $a_j$ in $Y\cap V_j$ (which by 
the definition of  $ \mathcal V $ is   nonempty), set $F=\{ a_j : 0²j²p\}$ and write  $V=\bigcup_nW_n$ with 
${W_0}\supset F$ and for all $n$, $\adh{W_n}\subset V$. It follows   that  for each $n$, 
 $\adh{W_n}\cap Y\in \FF_0^*(Y)$  so $\dim(\adh{W_n}\cap Y)=0$ hence by Theorem \ref{dim} 
 $\dim(V\cap Y)=0$.  Then    $U'= U\cup (V\cap Y)=U\cup V$  is an open subset of $X$ 
 and since   $\dim(U)=0$ then again  by Theorem \ref{dim}  $\dim(U')=0$. It follows then from  the definition of $U$  that  $V\cap Y=\0$ which is a contradiction. Hence $Y=\0$ and  $X=U$ so $\dim(X)=0$.
\end{proof}

\end{section}


\begin{section}{When $\FF_0(X)$ is  of low complexity}\label{low}

 It follows from the previous results that at least   in the frame of Polish spaces, aside in the trivial case of a zero-dimensional space,  in which case $\FF_0(X)=\FF(X)$,  the first two possible  descriptive complexity classes  for the set $\FF_0(X)$ are the class  $\Gd$ and the class $\ca$. In this  section we   show that both of these classes can really occur.  One of the main goals of this work is to prove the existence of other  possible classes.
  
 We first mention the following elementary proposition concerning the class $\Gd$. We point out that in Section \ref{Hypergraphs}  we shall give
an example of a Polish space $X$ which does not satisfy the assumptions of this proposition but  for which the set   $\FF_0(X)$  is   $\Gd$.
 
\begin{proposition} \label{F0compact} If  a space $X$ is decomposable into  a countable union $X=\bigcup_n F_n$  where  each $F_n$ is either compact or a closed zero-dimensional subset of $X$  then $\FF_0(X)$ is  a  $\Gd$ subset of $\FF(X)$. 
\end{proposition}

\begin{proof} Fix a countable basis $(U_n)_{n\geq 0}$ of the topology of $X$ and   observe that if a space $E$ is zero-dimensional then any pair $(A,B)$ of disjoint closed subsets of $E$ can be separated by a clopen set.
 Hence for any $F\in \FF(X)$ we can write:
 \begin{align*}
  F\in \FF_0(X) 
& \iff  
   \left\{\begin{array}{l} 
  \forall m, n :\ \adh{U_m}\subset U_n , \   \exists \,  V  \text{ open}  :\\
 \adh{U_m}  \subset   V\subset  U_n \  \ \wedge  \ \ F  \cap \partial{V}=\0\\
 \end{array} \right. \\
   & \iff   \left\{\begin{array}{l} 
  \forall m, n :\ \adh{U_m}\subset U_n , \   \exists \,  V  \text{ open} :\\
 \adh{U_m}  \subset   V\subset  U_n \  \ \wedge  \ \ F   \subset  V \cup \adh{V}^c \\
 \end{array} \right. 
 \end{align*}

a)  We shall first treat the case where $X$ itself is compact.  In this case   the   
condition $F   \subset  V \cup \adh{V}^c $  is open, hence by the equivalence above $ \FF(X)$ is $\Gd$.

b) Coming back to the general case, suppose that for all $k$, $F_{2k}$ is compact and $F_{2k+1}$ is zero-dimensional. Then  by   Theorem \ref{dim}  for any $F\in \FF(X)$  we have:
\begin{align*}
F\in \FF_0(X) & \iff  \forall k, \   F\cap F_{2k} \in \FF_0(F_{2k}) 
 \end{align*} 
 Since   $F_{2k}$ is compact then by case a) $\FF_0(F_{2k})$ is $\Gd$ and it follows from  {\ref{F(X)}.b)} that  the mapping $F\mapsto F\cap F_{2k}$ is Borel from  $\FF(X)$ to $\FF(F_{2k})$. Hence $\FF_0(X)$ is Borel and then  by  Corollary \ref{KLWF(X)}  $\FF_0(X)$ is  a $\Gd$ subset of $\FF(X)$.
 \end{proof}

As pointed out in the introduction, E. Pol and R. Pol showed in \cite{pol} that 
the set $\FF_0(\PP\times \I)$ is $\ca$ and not $\ana$.  
We shall prove in the rest of this section a generalization of  this  result, with a different proof.

 \begin{proposition}\label{prop}   If $X$ is  a closed subset of $Z \times K$ where  $Z$ is zero-dimensional and $K$ is compact  then the following are equivalent:
 \begin{listeroman} 
  \item     $\dim(X) =0$.
 \item    For all $z\in Z$,  $\dim(X(z))=0$.                                                                                                                                                                                                                                                                                                                                                                                                                                                                                                                                                                                                                                                                                                                                                                                                                                                                                                                                                     \end{listeroman} 
In particular if $Z$ is analytic then $\FF_0(X)$ is $\ca$.
 \end{proposition} 
 
  \begin{proof} 
  $(i) \impl (ii)$   is obvious. 
  
  $(ii) \impl (i)$: Fix $(z,y)\in X$ and $W$   a neighbourhood of $(z,y)$. Since $\dim(X(z))=0$ there  exists an open neighbourhood $V$ of $y$ in $K$ such that $V\cap X(z)$ is clopen in $X(z)$ and  $V\cap X(z)\subset W(z)$  hence we can find  $ W_0$ and $W_1$  two disjoint open subsets of $X$,  such that $ W_0\subset W$ , $W_0(z) = V$ and $W_1(z)= X(z)\setminus V$. 

Since $K$ is compact then $\pi$ the projection mapping  on $Z$  is closed hence the set $H=\pi( X\setminus (W_0\cup W_1))$ is closed,  and since  $z\notin H$  there exists a clopen set $U$ such that $z\in U\subset Z\setminus H$, hence  $V=  W_0\cap \pi^{-1}( U)$ is clopen in $X$ and   $(z,y)\in V\subset W$.

This proves the equivalence, and applying this equivalence to  $F\in \FF(X)$ we can write:
$$F\in \FF_0(X)  \iff   \forall z\in Z ,  \ F(z)\in \FF_0(K)$$
equivalently:
$$F\not\in \FF_0(X)  \iff   \exists z\in Z ,  \ L\not\in \FF_0(K): F(z)=L$$
Since $K$ is compact one easily  checks  that  the set  
 $\{(F,z, L)\in  \FF(X) \times Z\times \FF_0(K):   F(z)=L\}$  
 is a $\Gd$ subset of $\FF(X) \times Z\times\FF_0(K)$. Then since $Z$ is analytic and by  Proposition  \ref{F0compact} $\FF_0(K)$ is Polish, it follows then from the second equivalence above  that the set $\FF(X)\setminus\FF_0(X)$ is $\ana$ hence the set $\FF_0(X)$ is $\ca$.  
\end{proof}

\begin{proposition}If $X=Y\times Z$ is the product of  a $\s$-compact space  $Y$ and an analytic zero-dimensional space $Z$ then   $\FF_0(X)$ is   $\ca$.                                                                                                                                                                                                                                                                                                                                                                                                                                                                                                                                                                                                                                                                                                                                                                                                                                                                                                                                                           \end{proposition} 
 \begin{proof}  Set  $X=\bigcup_n F_n$ where for all $n$, $F_n=K_n\times Z$ with $K_n$ compact. Then by Theorem \ref{dim} we have for $F\in \FF(X)$:
$$F\in \FF_0(X)  \iff \forall n, \ F\cap F_n\in \FF_0(F_n) $$
Since  the mapping $F\mapsto F\cap F_n$ is Borel the conclusion follows then from Proposition \ref{prop}.  
\end{proof}

Observe that in the following theorem if  $Y$  were zero-dimensional then $X$ would be zero-dimensional too, and then  $\FF_0(X)=\FF(X)$ would be $\ana$, even $\Gd$ if $Y$ and $Z$ are Polish.
Also if $Z$  were $\s$-compact  then by Proposition \ref{F0compact}   $\FF_0(X)$ would be $\Gd$.

 \begin{theorem}   If $X=Y\times Z$ is the product of 
   a non zero-dimensional  analytic  space $Y$ and     a zero-dimensional  non  $\s$-compact analytic space  $Z$    then    $\FF_0(X)$                                                                                                                                                                                                                                                       is   $\ca$-hard.
 \end{theorem}

 \begin{proof} 
 Fix a compactification  $\hat Y$  for $Y$, and a zero-dimensional compactification $\hat Z$ of  $Z$. 
 Fix   also an  enumeration  $(y_n)_{n}$ of some countable dense subset  of $Y$.
 Since $Z$  is not $\s$-compact  then  by the classical  Hurewicz  Theorem we can find  a copy of the Cantor set 
  $C\subset \hat Z$   such that   
 $D=C\setminus  Z$  is a dense countable  subset of $C$.  We then fix   an  enumeration 
 $(\hat z_k)_{k}$ of $D$ and  for all $k, \ell$   a     clopen neighbourhood $W_{k,\ell}$ of $\hat z_k$  in $ \hat Z$ of diameter  $< 2^{-k-\ell}$ and an element   $z_{k,\ell}\in W_{k,\ell}\cap Z$. Finally let $(k,\ell)\mapsto \<k,l\>$ denote any bijection from $\wo^2$ to $\wo$. 
  
  Let $\KK(C)$ the (compact) space of all compact subsets of $ C$. For any  $K\in\KK(C)$  set: 
  $$\Phi(K)=  \{(y_{\<k,\ell\>}, z_{k,\ell}) : \ (k,\ell) \text { such that } W_{k,\ell}\cap K\not=\0\}\subset Y\times Z $$
and
$$\Psi(K)= \hat Y\times K  \, \cup\, \Phi(K) \subset \hat Y\times \hat Z\,. $$

\begin{claim}  $\adh{\Gr(\Phi)}\subset \Gr (\Psi)$.
\end{claim}
\begin{proof}
Suppose that $(K,y,z)=\lim_j(K_j,y_{\<k_j,\ell_j\>}, z_{k_j,\ell_j})_j$ with 
$(y_{\<k_j,\ell_j\>}, z_{k_j,\ell_j})\in \Phi(K_j)$. Then   for all $j$ there exists some $z_j\in W_{k_j,\ell_j}\cap K_j$,
so  {$z_j\in   K_j$}  and since $\diam (W_{k_j,\ell_j})< 2^{-k_j-\ell_j}$  then $d (z_{k_j,\ell_j}, z_j)< 2^{-k_j-\ell_j}$.

\nd -- If $\lim(k_j+\ell_j )=\infty$ then $\lim_jd (z_{k_j,\ell_j}, K_j)=0$ hence $z\in K$ and $(y,z)\in  \Psi(K)$.

\nd -- If not  then for infinitely many
 $j$'s the sequence $({k_j,\ell_j})_j$ is constant, with value  say  $(k,\ell)$, so 
 $(K_j,y_{\<k_j,\ell_j\>}, z_{k_j,\ell_j})=(K_j, y_{\<k,\ell\>}, z_{k,\ell})$ with $W_{k,\ell}\cap K_j\not=\0$; and since
 $W_{k,\ell}$ is closed  then $W_{k,\ell}\cap K\not=\0$ hence   $(y,z)\in \Phi(K)\subset \Psi(K)$. 
\end{proof} 

It follows from the claim that
$ ( \hat Y\times K)\cup \adh{\Gr(\Phi)} \subset \Gr(\Psi) \subset ( \hat Y\times K) \cup \Gr(\Phi) $ 
 hence $\Gr(\Psi)=( \hat Y\times K)\cup \adh{\Gr(\Phi)}$ is a closed  subset of  $\KK(C)\times \hat Y\times C$
and the  compact-valued mapping $K\mapsto \Psi(K)$  is   u.s.c..   Hence for any open set $U\subset \hat Y\times C$ the set $\{K\in\KK(C) : \ \Psi(K)\subset U\}$ is open.
Moreover if   $U=V\times W$ is a  basic open  set then
$$\Psi(K)\cap U\not=\0\iff \left\{
\begin{array}{l}  \exists\,({k,\ell}) , \  (y_{\<k,\ell\>}, z_{k,\ell}) \in U \text{  and } K\cap W_{k,\ell}\not=\0 \, ,\\
\text{or } K\cap W\not=\0 
 \end{array}\right.$$
so  the set $\{K\in\KK(C) : \ \Psi(K)\cap U\not=\0\}$ is open too. Hence the mapping  $\Psi: \KK(C)\to \FF(\hat Y\times C)$ is continuous. 

Observe that   $\hat X=\hat Y\times\hat Z$ is a compactification of $X=Y\times Z$ and for all $K$, $\Psi(K)\cap  X$ is a closed subset of $X$ which by the lemma is dense in $\Psi(K)$. Hence if we identify $\FF(X)$ to a subspace of $\FF(\hat X)$ then the mapping $\Psi$ is identified to the mapping  $\td\Psi:  \KK(C)\to \FF(X)$ with $\td \Psi (K)=\Psi(K)\cap  X$ which is then continuous for the corresponding topology. Observe finally:

-- if $K\in \KK(D)$  then $\td \Psi (K)=\Phi(K)$  is countable hence $\td \Psi (K)\in \FF_0(X)$ 

-- if not and $z\in K\setminus D$ then $\td\Psi (K)\supset Y\times \{z\}$  hence $\td \Psi (K)\not\in \FF_0(X)$ 

\nd Hence $\td \Psi$ is a continuous reduction of  $\KK(D)$ to  $\FF_0(X)$, and it is well known that  if $D$ is   any dense  countable subset  of the Cantor space then $\KK(D)$ is a  $\ca$-complete  subset of  $\KK(2^\wo)$,   hence  $\FF_0(X)$ is $\ca$-hard.
 \end{proof}

\begin{corollary}\label{IxP}  If $X=Y\times Z$ is the product of  a  $\s$-compact non zero-dimensional   space  $Y$ and an analytic  zero-dimensional non $\s$-compact    space $Z$ then   $\FF_0(X)$ is $\ca$-complete.                                                                                                                                                                                                                                                                                                                                                                                                                                                                                                                                                                                                                                                                                                                                                                                                      \end{corollary}

 \begin{corollary}\label {erdos} If $X$ is  a   non zero-dimensional  and non $\s$-compact  analytic  space   homeomorphic  to its square $X^2$   then  $\FF_0(X)$  is    $\ca$-hard.                                                                                                                                                                                                                                                 \end{corollary}
 
 \begin{ssection} {\bf The Erd\H os space:}
 {\rm As an application we mention the case of  Erd\H os space $S$ (\cite{erd}). We recall   
 that this is the  closed subset  $S$ of the Hilbert space $\ell^2=\ell^2(\mathbb N)$ defined by: 
 $$S=\{(x_n)_{n\geq 0}\in\ell^2: \forall n, \ x_n=0 \ \text{ or }  x_n^{-1}\in \mathbb N\}.$$
We claim that  the set $\FF_0(S)$ is $\ca$-hard. Indeed since   any compact subset of $S$ is zero-dimensional then by 
Theorem \ref{dim} the Polish space $S$ is non $\s$-compact.                                                                                                                                                                                                                                       Moreover the isometry 
$\Psi : \ell^2\times \ell^2\to \ell^2$ defined by $\psi(e_k, 0) =e_{2k}$ and $\psi(0,e_k) =e_{2k+1}$ 
induces an homeomorphism from $S^2$ onto $S$, hence by Corollary \ref{erdos} $\FF_0(S)$ is $\ca$-hard.
However we do not know the exact complexity of   $\FF_0(S)$.
}
 \end{ssection}

\end{section}

\begin{section}{The maximum complexity of $\FF_0(X)$ when $X$ is $\ana$}   \label{maxpca}

  \begin{proposition} \label{F0=S12}  If $X$ is an analytic  space then $\FF_0(X)$ is    $\pca$. 
\end{proposition} 
\begin{proof} Fix a countable basis $(U_n)_{n\geq 0}$ of the topology of $X$. 
Resuming the analysis in the proof of  Proposition  \ref{F0compact} we have  for any $F\in \FF(X)$:
\begin{align*}
  F\in \FF_0(X)  
  & \iff   \left\{\begin{array}{l} 
  \forall m, n :\ \adh{U_m}\subset U_n , \   \exists \,  V  \text{ open} :\\
 \adh{U_m}  \subset   V\subset  U_n \  \ \wedge  \ \ F   \subset  V \cup \adh{V}^c \\
 \end{array} \right. \\
& \iff   \left\{\begin{array}{l} 
  \forall m, n :\ \adh{U_m}\subset U_n , \   \exists \,  V  \text{ open} :\\
 \adh{U_m}  \subset   V\subset  U_n \  \ \wedge  \ \ F   \cap  V^c \cap \adh{V}=\0 \\
 \end{array} \right. \\
  & \iff \left\{\begin{array}{l} 
  \forall m, n :\ \adh{U_m}\subset U_n , \   \exists \,  F'\in \FF(X):    \\
U_n^c \subset  F' \subset \adh{U_m}^c \  \ \wedge  \ \ F  \cap F'\cap  \bigl({{\rm Int} (F')}\bigr)^c=\0\\
 \end{array} \right.
 \end{align*}
Then  observe that by  {\ref{F(X)}.c)}  condition  ``$U_n^c \subset  F' \subset \adh{U_m}^c $" on $F'$ is $\ana$   while      condition ``$F  \cap F'\cap \bigl({{\rm Int} (F')}\bigr)^c=\0$" on $(F,F')$  is co-analytic since the dual condition 
``$F  \cap F'\cap \bigl({{\rm Int} (F')}\bigr)^c\not=\0$" is equivalent to:
$$\exists x\in  F\cap F' : \ \forall p  \  \bigl(x\not\in U_p \   \vee    \     (\exists q: \  {U_q}\subset U_p \ \wedge \  F'\cap U_q=\0)\bigr)$$
which is  analytic by  {\ref{F(X)}.b)}. It follows from the previous analysis that $\FF_0(X)$ is the projection of  the intersection of an analytic set with a co-analytic set, hence $\FF_0(X)$ is  $\pca$.
\end{proof}

\begin{theorem} \label{F0=S12-complete} There exists an analytic  subspace $X$ of $2^\wo\times \I$ for which  $\FF_0(X)$ is     \hbox{$\pca$-complete.} 
\end{theorem} 
\begin{proof}  
 Fix  an enumeration $(I_n)_n$ of all  nonempty  subintervals of  $\I$ with endpoints in $\Q$, and  fix for all $n$,  
  $G_n\subset I_n$   a copy of $ \wo^\wo$  and a  homeomorphism $h_n:G_n\to \wo^\wo$.

 Let $A$ be any $\pca$ subset of $2^\wo$,  that we represent as the projection on the first factor of some $\ca$ set  $B\subset 2^\wo\times  \wo^\wo$, and consider the set:
$$X=\{(\a, z)\in 2^\wo\times \I: \ \forall n, \ z\not\in G_n \vee (z\in G_n \wedge  (\a, h_n(z))  \not \in B)\}
$$ 
which  is clearly $\ana$.

\begin{claim}\label{lem} $\a\in A$   if and only if $\dim (X(\a))=0$.   
\end{claim}

\begin{proof}
-- If $\a\in A$   fix $\b\in \wo^\wo$ such that $(\a,\b) \in B$ and set  for all $n$, $z_n=h_n^{-1}(\b)\in G_n$ then 
$(\a, h_n(z_n))=(\a,\b)\in B$ hence $(\a, z_n)\not \in X$. Since the set  $D= \{z_n; \   n\in \wo\}$ is dense in $\I$ then  the set $E=\I\setminus D$  is of empty interior in $I$  hence $\dim(E)=0$  and since $X(\a)\subset E$ then  
$\dim(X(\a))=0$.

-- If $\a\not\in A$ then for all  $\b\in \wo^\wo$  $(\a,\b)\not \in B$ hence for all $z\in \I$,  if  $z\in G_n$ then  $(\a, h_n(z))\not \in B$, hence $X(\a)\supset \I$ hence $\dim(X(\a))\geq 1$. 
\end{proof}
 
It follows then from  Claim~\ref{lem} and Corollary \ref{E0} that the set $A$ is reducible to $\FF_0(X)$. 
So starting with a $\pca$-complete set $A$ and considering  the corresponding  space $X$,    the set  $\FF_0(X)$  is then $\pca$-hard, hence by Proposition \ref{F0=S12}, $\pca$-complete.
\end{proof}

\end{section}


\begin{section}{The maximum descriptive class of $\FF_0(X)$ when $X$ is Borel}  \label{maxborel}

We do not know whether Theorem  \ref{F0=S12-complete} holds for some Polish  space $X$.
We are only able to prove that in this more restrictive context the maximum complexity of $\FF_0(X)$ is much beyond the class $\ca$. More generally we  shall   prove a  similar result  for the case where $X$ is a $\borm \xi$ set, with a lower bound increasing with $\xi$, still   below the class $\boldsymbol{\Delta}^1_{2}$.

\begin{ssection} {\bf The  classes $\Game \G$:}\label{GG}
\medskip

\nd {\bf \ref{GG}.a} {\it Games and strategies:} {\rm 
Given any set $A\subset \wo^\wo$ we denote by $\bG_A$ the standard game on $\wo$ with $A$ as the win condition for Player I.
More explicitly  $\bG_A$  is the game in which the two Players choose alternatively an element 
in $\wo$:

$\begin{array}{ccccccc}
{\rm I:} & n_0&      &n_2&      &\cdots&         \\
{\rm II:}&       &n_1&      &n_3&          &\cdots 
\end{array}$

\nd Thus an  infinite run in the game is  an infinite  sequence $\g=(n_0, n_1,n_2 \cdots)\in \wo^\wo$
and   Player I wins the run if  $\g\in A$.

We view a strategy  in such a game as a  mapping $\s:\wo^{<\wo}\to \wo$ (for Player I ) or  $\tau: \wo^{<\wo}\setminus \{\0\}\to \wo$ (for Player II). We denote by $\Sigma_{\rm I}$ ($\Sigma_{\rm II}$) the set of all strategies for Player I (Player II) that we endow with their natural product topologies for which 
$\Sigma_{\rm I}\approx\Sigma_{\rm II}\approx\wo^\wo$. 

Given any pair 
$(\s,\tau)\in\Sigma_{\rm I}\times\Sigma_{\rm II}$  we define  inductively the infinite run 
$$\s\star\tau=\g=(n_0, n_1,n_2 \cdots)\in \wo^\wo$$ 
by $n_0=\s(\0)$ , $n_{2k}=\s(n_1, n_3, \cdots,n_{2k-1})$  and   $n_{2k+1}=\tau(n_0, n_2, \cdots,n_{2k})$. 
The mapping $\star: \Sigma_{\rm I}\times\Sigma_{\rm II}\to \wo^\wo$ thus defined
is then clearly continuous and:

-- the strategy $\s\in\Sigma_{\rm I}$ is  winning   in the game $\bG_A$ iff for all $\tau \in\Sigma_{\rm II}$, $\s\star\tau\in A$

-- the strategy $\tau\in\Sigma_{\rm II}$ is   winning  in the game $\bG_A$ iff for all $\s \in\Sigma_{\rm I}$, $\s\star\tau\not\in A$

\medskip

\nd {\bf \ref{GG}.b} {\it The game operator  $\Game$:}
{\rm Let     $B\subset X\times \wo^\wo$ where $X$ is an arbitrary Polish space then for all $x\in X$, $B(x)\subset \wo^\wo$ and one can consider the game $\bG_{B(x)}$, and then define 
$$\Game B=\{x\in X: \ \text{ Player I  wins the game } \bG_{B(x)}\}$$

In all the sequel $\G$ denotes a non trivial class  $(\G$ and $\check \G\not=\{\0\})$ which admits  a $\G$-complete set  (that is $\G$ is a Wadge class). In particular $\G$ could be  any  Baire  class or   projective class. 

\medskip

Given such a  class $\G$, a set $A\subset  X$ is said to be in $\Game \G$ if it is of the form 
$A= \Game B$ for some $B\subset X\times \wo^\wo$ in $\G$.  The reader can find in \cite{mo} a detailed study  of
the operator  $\Game$. We state next  only some   elementary properties:
\nd {\it
\begin{listeroman}
\item  ${\Game \G}$ is closed under  countable unions and intersections.
\item If $\G$ is closed under  unions with   closed sets  then  ${\Game \G}$ is closed under  
\\ operation $\A$.
\item If all games in $\G$ are determined then $\Game\check \G$ is the dual class of  ${\Game \G}$.
\item If  $\G\subset \boldsymbol{\Pi}^1_n$ then ${\Game \G}\subset\boldsymbol{\Sigma}^1_{n+1}$ 
and \, if  $\G= \boldsymbol{\Pi}^1_n$ then ${\Game \G}= \boldsymbol{\Sigma}^1_{n+1}$
\end{listeroman}
} 

 Observe that  since all Borel  games are determined then ${\Game \borel} \subset \boldsymbol{\Delta}^1_2$ with   strict inclusion in general. In fact it turns out that sets in ${\Game \borel}$ possess all classical descriptive regularity properties such as Baire property or  universal measurability,  while  for an arbitrary    $\boldsymbol{\Delta}^1_2$ set  such a property might depend on  Set Theory axioms.

  It is easy to see  that $\Game \bora 1=\ca$  and  $\Game \borm 1=\ana$ (see 6D.2 in \cite{mo} edition 2009), and   for Baire classes $\G$ of low rank   the classes ${\Game \G}$ coincide with some more classical classes   considered  much before the intrusion of games in Descriptive Set Theory. In particular the classes $\Game \bora 2$ and $\Game \bora 3$ are related to Kolmogorov's $R$-sets (we refer the reader to  \cite{bg} and \cite{bg2} for definitions and  more details). Also relying on the classical transfinite analysis of the class $\bord 2$ one easily derive from \ref{GG}.b.{\it (ii)} that  $\Game \bord 2$ is  the smallest non trivial class closed under
complementation and   operation $\A$, also known as Selivanovski's class of  $C$-sets.

We also point out that by a result due to Solovay   $\Game \bora 2$ is  exactly the class of sets  admitting a $\ana$-inductive definition  (see \cite{mo}). This probably explains why $\Game \bora 2$ sets arise naturally in analysis,  though   very few   natural  examples of $\Game \bora 2$-complete sets are known, while no such examples are known for higher $\Game \G$ classes. We mention the following two simple examples:
 
 \medskip
 
   {\it

\nd-- {\sc (Louveau \cite{lv}):} There exists a $\Gd$ set $G\subset 2^\wo\times \wo^\wo $ such that if $\pi$ denotes the projection mapping on the first coordinate and $\mathcal B=\{\pi(K); \ K\in \KK(G)\}$ then  the $\s$-ideal  of $\KK( 2^\wo )$ generated by   $\mathcal B$ is \hbox{$\Game \bora 2$-complete.} 

\nd --  {\sc (Saint Raymond \cite{sr}):} 
Let ${\mathcal T}\subset 2^{\wo^{<\wo}}$ denote the set of all trees on $\wo$.  Then the set 
 ${\mathcal Q\mathcal B}=\{T\in \mathcal T:   \exists \a\in\wo^\wo, \ \forall \b\in \lceil T
\rceil \ \a\not\leq \b\}$ (where $\leq$ denotes the product order on $\wo^\wo$)  is \hbox{$\Game \bora 2$-complete.}
 }
 }}
\end{ssection}

The following well known result is the only result we will use in this section. For the convenience of the reader
we give the simple argument which is based on the fundamental notion of ``good universal sets"  which play a central role in recursivity 
(see \cite{mo}).
 
\begin{lemma} \label{Univ} For any Baire class $\G$  there exists a   set 
$U\subset 2^\wo\times\wo^\wo$   in $\G$ such that the set $\Game U$ is $\Game\G$-complete. 
\end{lemma}

\begin{proof} Fix a homeomorphism $\ph: 2^\wo \times 2^\wo\to2^\wo$
and  for all $\a$ let $\ph_\a$ denote the partial mapping $\ph(\a,\cdot):2^\wo\to2^\wo$  .
Starting from any  $\G$-set $V\subset 2^\wo\times(2^\wo\times\wo^\wo)$  which is 
universal for  $\G$-subsets of  $(2^\wo\times\wo^\wo)$ define
$$U=\{(\ph(\a,\b), \g)\in 2^\wo\times\wo^\wo:\  (\a,\b, \g)\in   V\}$$
Then for any $\G$-set $B\subset 2^\wo\times\wo^\wo$ there exists some $\a\in 2^\wo$ such that  
$B=V(\a)$ so  for any  $(\b, \g)\in 2^\wo\times\wo^\wo$ we have:
$$(\b, \g)\in B\iff  (\a,\b, \g)\in   V\iff (\ph(\a,\b), \g)\in U\iff  (\ph_\a(\b), \g)\in U$$
hence  $U$ is   $\G$-complete. Moreover  for any $\b\in 2^\wo$, since
$B(\b)=U(\ph_\a(\b))$, then: 
\begin{align*}
 \b\in \Game B &\iff  \text{ Player I wins }   \bG_{B(\b)} \\
&\iff  \text{ Player I wins }  \bG_{U(\ph(\a,\b))}\iff
\ph_\a(\b)\in \Game U
\end{align*}
hence $\ph_\alpha$ reduces $\Game B$ to $\Game U$ and
$\Game U$ is   $\Game\G$-complete. 
\end{proof}

\begin{theorem} \label{F0=GSxi-complete} Let $H$ be a perfect non zero-dimensional Polish space and $G$ a dense $\Gd$ subset of $H$. Assume that for every dense subset $D$ of $G$ the space $ H\mns D$ is  zero-dimensional.
Then for any countable ordinal $\xi\geq 3$ there exists a $\borm \xi$  set $X\subset  {\wo^\wo} \times H$ such that   
$\FF_0(X)$ is      $\Game \bora \xi$-hard. 

 If moreover $G$ is nowhere relatively $\sigma$-compact in $H$ then the same conclusion holds for $\xi=2$. 
\end{theorem}

\begin{proof} 
 Fix   a countable  basis $(W_n)$  of the topology in $H$. Since  $G$ is dense in $H$ then   for all $n$ the set $G\cap W_n$ is perfect so we can fix  a    copy of $\wo^\wo$  say  $G_n$ in $G\cap W_n$  and      a homeomorphism  $h_n:G_n\to \Sigma_{\rm I}\approx \wo^\wo$.

Given any Borel set $U \subset 2^\wo\times\wo^\wo$ let
$$X_U=\{(\a,\tau, z)\in 2^\wo\times\Sigma_{\rm II}\times H: \  
\forall n, \ z\in G_n \rightarrow  (\a, h_n(z)\star\tau)  \not \in U\}$$

\begin{claim}\label{lemxi}     $\a\in \Game U$  if and only if  $\dim(X_U(\a))=0$. 
\end{claim}

\begin{proof}
-- If $\a\in  \Game U$ then Player I has a winning strategy in the game $\bG_{U(\a)}$. Fix such a strategy 
$\s^*\in\Sigma_{\rm I}$ and let  for all $n$, $z_n=h_n^{-1}(\s^*)\in G_n$ then for 
all  $\tau\in\Sigma_{\rm II}$, $h_n(z_n)\star\tau=\s^*\star \tau\in U(\a)$  hence $(\a, \tau,z_n)\not \in X_U$. Then  by construction   the set  $D= \{z_n: \   n\in \wo\}$ is dense in $G$ hence      the set  $E=H\setminus D$ is zero-dimensional. But  by the observation above $X_U(\a)\subset \Sigma_{\rm II}\times E$ hence $\dim(X_U(\a))=0$.

-- If $\a\not\in  \Game U$ then, since the game  $\bG_{U(\a)}$ is determined, Player II has a winning strategy. Fix such a strategy  
 $\tau^*\in \Sigma_{\rm II}$; then  for all $z\in H$,  if  $z\in G_n$  then $\s_n=h_n(z)\in\Sigma_{\rm I}$,  and since 
 $\tau^*$ is winning  for Player II   then
 $ h_n(z_n)\star\tau^*=\s_n\star \tau^*\not\in U(\a)$  
hence $X_U(\a,\tau^*)\supset H$ and so $\dim(X_U(\a))\geq \dim(X_U(\a,\tau^*))\geq \dim(H)>0$. 
\end{proof}

From now on   we let  $U \subset 2^\wo\times\wo^\wo$ be a  $\bora \xi$  set  given by  Lemma \ref{Univ} and set
$X=X_U$.  Then as in the proof of Theorem \ref{F0=S12-complete}, applying 
Corollary \ref{E0}, one   concludes that the set $\FF_0(X)$ is $\Game\bora\xi$-hard. 

To finish the proof we only need to check that $X$ has the expected complexity. 
Observe first that since $2^\wo\times\Sigma_{\rm II}\approx 2^\wo\times\wo^\wo\approx \wo^\wo$ then $X$ can indeed  be viewed as a subset of  $\wo^\wo\times H$. Moreover for a given $n$:
$$\bigl(z\in G_n \rightarrow  (\a, h_n(z)\star\tau)  \not \in U\bigr) \iff 
 \bigl(z\not\in G_n  \ \vee \  (z\in G_n\wedge (\a, h_n(z)\star\tau)  \not \in U)\bigr)$$
Since $U$ is     $\bora \xi$   and since $h_n:G_n\to 2^\wo\times\wo^\wo$ is continuous and its domain $G_n\approx \wo^\wo$ is  a $\borm 2$ set then  
 for all $\xi\geq 2$ the condition ``$z\in G_n\wedge (\a, h_n(z)\star\tau)  \not \in U$''  on $(\a,\tau, z)\in 2^\wo\times\Sigma_{\rm II}\times H$  is   $\borm \xi$. Moreover  if $\xi\geq 3$   the condition ``$z\not\in G_n$" on $z\in H$ defines a  $\bora 2$ subset of $H$, hence an (absolute) $\borm \xi$ set.
This  proves the first part of the theorem.

\sk

If $G$ is nowhere relatively $\s$-compact in $H$ then no  $ W_n\cap G$ is contained in a $\s$-compact subset of $H$. Hence by Hurewicz' theorem   for all $n$ there exists a  subset $ G_n$ of $W_n\cap G$ which is closed in $H$ and homeomorphic to $ \wo^\wo$.
  Then resuming the complexity analysis of $X$  the last condition ``$z\not\in G_n$"  defines now an open subset of $H$  hence an (absolute) $\borm 2$ set, and it follows that  if $\xi=2$ then $X$ is a $\borm 2$ set.                                                           
\end{proof}

 \begin{corollary} \label{F0=GSxi=3-complete} For any   {positive countable ordinal $\xi\not =2$} there exists a $\borm \xi$  set $X\subset \wo^\wo\times \I$ such that  $\FF_0(X)$ is      $\Game \bora \xi$-hard. 
\end{corollary}

\begin{proof}  {For   $\xi=1$, $\Game \bora 1=\ca$ and the result  follows from Corollary \ref{IxP}   with $X=\PP\times \I$;   and  for $\xi\geq 3$  the result  follows from Theorem  \ref{F0=GSxi-complete}, with $H=G=\I$.}
\end{proof}

 We do not know whether Corollary \ref{F0=GSxi=3-complete} holds for $\xi=2$, but  in  Section \ref{maxpol}
 we shall construct a    $\borm 2$  subset of  $\wo^\wo\times \I^2$     for which the set $\FF_0(X)$
 is  $\Game \bora 2$-hard.


 \end{section}
 \begin{section}{Hypergraphs}\label{Hypergraphs}

   Given any  function $f: P \subset \R \to \R $ and any element $a\in \adh P$   we denote by:
\begin{align*}
 \Omega f (a) & :=\{\lim_jf(x_j):  \  x_j\in P \ \text{ and }\lim_j x_j=a\}  \\
 \Omega^- f (a)& :=\{\lim_jf(x_j):  \  x_j\in P , \ x_j\leq a \ \text{ and }\lim_j x_j=a\}\\
\Omega^+ f (a) &: =\{\lim_jf(x_j):  \  x_j\in P, \ x_j\geq a \  \ \text{ and }\lim_j x_j=a\}
\end{align*}
 the sets of all (left, right) cluster values of $f$ at $a$, so   $ \Omega f (a)= \Omega^- f (a)\cup \Omega^+ f (a)$. Observe   that if $G$ is the graph of $f$ and $ \adh G$ its closure in $\R ^2$ then 
 $\Omega f (a)= \adh G(a)$.

\begin{lemma} \label{hyper}   Given any  countable  subset $Q$ of  a non trivial  interval $I\subset \R $
there exists a continuous function $f:P= I \setminus Q \to \R $   satisfying for all $a\in Q$:

a)  $\Omega f (a) =\Omega^- f (a) =\Omega^+f (a)$ is a non trivial interval.

b)  There  exists a  continuous function $g:I\setminus\{a\} \to \R $ such that the   function

\quad    $f-g: P \to \R $   admits a limit   at $a$.
 \end{lemma} 
 
 \begin{proof} Fix an enumeration $(a_n)_{n\geq 0}$ of  $Q$ and  set for all $n$, $g_n(x)=  2^{-n}\sin (\dfrac 1 {a_n-x})$. Then  $g_n: I\setminus\{a_n\} \to [-2^{-n} ,2^{-n} ]$  is  a continuous function and  
 $$\Omega g_n (a_n) =\Omega^- g_n (a_n) =\Omega^+g_n (a_n)=[-2^{-n} ,2^{-n} ].$$
Finally set  for $x\in P= I \setminus Q$:
$$f(x)= \sum_{n=0}^\infty  g_n(x).$$
Since  $ \abs{g_n(x)}\leq  2^{-n}$ for all $x\in P$ the series   is  uniformly convergent on $P$, hence the function $f$ is well defined and continuous on  $P$. Similarly for all $n$, the function $h: P\cup \{a_n\}\to \R $
 defined by   $h(x)= \sum_{k\neq n} g_k(x)$ is also continuous, so  
 $f(x)-g_n(x) =h(x)$ tends to $h(a_n)$ when $x$ tends to  $a_n$ and
\begin{align*}
\Omega f (a_n) &=\Omega^- f (a_n) =\Omega^+f (a_n)=[h(a_n)-2^{-n} , h(a_n)+2^{-n} ].\qedhere
\end{align*}
\end{proof}

It was pointed to us by E. Pol and R. Pol that the function $f$ used in the previous proof was   considered  by K. Kuratowski and W. Sierpinski
in  \cite{ks}.

\begin{definition}\label{defhyper}  We shall say that  a set    
$H\subset \R ^2$  is  an hypergraph if   $G\subset H\subset\adh G$ 
where $G$ is the graph of a  continuous function  $f: I \setminus Q\to \R $   satisfying clauses a) and b) of \hbox{Lemma \ref{hyper}} and  for all $a\in Q$,  the set $H(a)$ is dense  and of empty  {interior in $\adh G(a)$.} 
 We shall then  refer to the set $G$ as the graph part of $H$. 
\end{definition}

\begin{ssection} {\bf Remarks:}\label{rem} 
{\rm 
a) If $H$ is an hypergraph   and $I, Q, P, f, G$  are as  in  Definition~\ref{defhyper} we shall  write 
 $H\equiv (H,I, Q, P, f, G)$. Observe that  this tuple  is uniquely determined by $H$ since $I$ is the projection of $H$ on the first factor, $Q=I \setminus P$,  and $P$ is the set of all elements $x\in I$ such that $H(x)$ is a singleton  $\{f(x)\}$, which  defines $f$ and  $G=\Gr(f)$. 

b) \red{If $I'$ and $J'$ are  intervals such that $I'\cap P\not=\0$ and $f(I'\cap P)\subset J'$ then clearly $H'=H\cap (I'\times J')$ is an hypergraph too.
It follows that  any nonempty open subset $W$ of an hypergraph $H$ contains an hypergraph. Indeed since the graph part $G$ is dense in $H$  one  can pick some element $(a,f(a))$ in $W\cap G$ and  then by continuity of $f$ at $a$  find   intervals $I'\ni a$ and $J'\ni f(a)$ such that $f(I'\cap P)\subset J'$ and  $H\cap (I'\times J')\subset W$.}

c) By Baire Theorem  an hypergraph cannot be an $\Fs$ set of $\R^2$: indeed
if $H$ were the union of countably many compact sets $F_n$ then there would be some $n$ such that $ G\cap F_n$ has nonempty interior in $G$ hence some subinterval $J$ of $I$ such that $\forall x\in P\cap J\ (x,f(x))\in F_n$. And we would have $ \adh G \cap ( J\times\R)\subset F_n\subset H$. In particular for $x\in Q\cap J$, $H(x)$ would contain an interval and would not be zero-dimensional. 

But  if $Q'$ is any countable  dense subset of $\R$  then $\adh G \setminus Q\times Q'$ is    a $\Gd$ hypergraph. In particular there are $\Gd$ hypergraphs  which are relatively closed subsets of $\I^2 \setminus \Q^2$.}   
\end{ssection}

\medskip

In all the sequel $(H, I, Q, P, f, G)$ denotes an arbitrary given hypergraph and  $\pi:\R ^2\to \R $ denotes the projection mapping on the first factor.

\begin{theorem} \label{Hconnex} \red{If $H$  is  an hypergraph and  $E \subset H$ is such that $\pi(E)$ is a non trivial interval then the set $E$
 is   connected.
 In particular any hypergraph  is connected.}
 \end{theorem}

 \begin{proof} \red{Replacing $I$ by the interval if necessary we may suppose that $\pi(E)=I$ (see Remark \ref{rem} b)).}
 So let  $U_0$ be an open subset of $ I\times \R $ such that $U_0\cap {\red E}\not=\0$ with   $\partial U_0\cap {\red E}=\0$  and set 
 $U_1= I\times \R  \setminus \adh {U_0}$. 
So $U_0\cap {\red E}$ and $U_1\cap {\red E}$ are two clopen subsets of ${\red E}$ and we shall prove that 
${\red E}\subset U_0$.

\red{Observe that since $\pi(E)=I$ then necessarily $E\supset G$} and set  $K={\adh G}$ (the closure of $G$ in $\R^2$). Since $K$ is a closed subset of $I\times \R$  with compact sections the set-valued mapping  $x\mapsto
 \td K(x)$ is u.s.c. (we recall that   $\td K(x)= \{x\}\times K(x) \subset K$)   hence  for $i\in \{0,1\}$ the set $V_i=\{x\in I: \ \td K(x) \subset U_i\}$  is open in $I$; and since $\td K(x)\not=\0$ for all $x\in I$ and  $U_0\cap   U_1=\0$ then $V_0\cap   V_1=\0$. Moreover since $G=\Gr (f) \subset  U_0\cup   U_1$   then  
$P \subset V_0\cup   V_1$. Hence $F=I \setminus (V_0\cup   V_1)\subset Q$ is a countable closed subset of $I$ and we shall denote, for any countable ordinal $\xi$,   by $F^{(\xi)}$  its  Cantor derivative of order 
$\xi$. 

 \medskip 
 
\begin{claim}\label{claim} For any  interval  $J\subset I \setminus F^{(\xi)}$ there exists $i\in \{0,1\}$ 
such that for all $x\in J$,   $\red{\td E}   (x) \subset U_i$.
\end{claim}
 
\begin{proof} We prove the Claim by induction on $\xi$:

 -- For $\xi=0$, $F^{(0)}=F$ so   $J\subset V_0\cup   V_1$  hence
$J\subset V_i$ for some $i$;  so for all $x\in J$, $\red{\td E}(x) \subset\td K(x)\subset U_i$.

 -- Assume now that the Claim  is proved for all $\eta<\xi$ and fix  an interval $J\subset I \setminus F^{(\xi)}$. Then  for all  $a\in J$ we can fix some $\eta_a<\xi$ such that $a\not\in F^{(\eta_a +1)}$ and  an open interval $J^a=]a-\e, a+\e[$ such that $J^a\cap F^{(\eta_a)}\subset \{a\}$. 
\medskip

 {\it  \underbar{Subclaim}: For all  $a\in  J$ there exists $i\in \{0,1\}$ and an open neighbourhood  $W$ of $a$  such that for all $x\in W$,  $\red{\td E}(x)\subset U_{i}$ .}
\medskip
 
 We first finish the proof of the Claim assuming   the  Subclaim: Set for   $i\in \{0,1\}$,
   $W_i=\{x\in J : \ \red{\td E}(x)  \subset U_i \}$. By the Subclaim   $(W_0, W_1)$ is an  open covering of $J$ and since  $\red{\td E}(x) \not=\0$  for all $x\in J$  and $U_0\cap   U_1=\0$ then $W_0\cap   W_1=\0$; hence  $J=W_i$ for some $i$, which proves the Claim for $\xi$.

  \vskip 2mm
 To prove the Subclaim we distinguish two cases:
\smallskip 

\noindent -- If $a\not\in F^{(\eta_a)}$ then the conclusion follows from  the induction hypothesis for  $\eta_a$.
\smallskip 

\noindent -- If $a\in F^{(\eta_a)}$  we distinguish three cases: 

  -- If $a$ is    a right  endpoint of $I$  then applying  the induction hypothesis for  $\eta_a$   with  $J=]a-\e, a[$ we can find  $i_{-}\in\{0,1\}$ such that
    $\red{\td E}(x)\subset U_{i_{-}}$ for all $x\in ]a-\e, a[$.  Now \red{pick} any $b\in {\red E}(a)\subset  \Omega f (a) =\Omega^- f (a)$; we can then find some sequence $(x_j)_j$  in $]a-\e, a[\cap P$ such that 
$(a,b)=\lim_j (x_j,f(x_j))$. Then  for all $j$,  $(x_j,f(x_j))\in  {\red E} \cap U_{i_{-}}$;  and since  ${\red E}\cap U_{i_{-}}$  is a closed subset of ${\red E}$   then $(a,b)\in  {\red E} \cap U_{i_{-}}$ too, hence  $\red{\td E}(a)\subset U_{i_{-}}$. 
 Similarly:
 
  -- If $a$  is a  left  endpoint of $I$ we can   find  $i_{+}\in\{0,1\}$ such that
   $\red{\td E}(x)\subset U_{i_{+}}$ for all $x\in [a, a+\e[$. 
 
  -- If $a$ is an interior point  we can find     $i_{-}$ and $i_{+}$ in $\{0,1\}$ such that  $\red{\td E}(x)\subset U_{i_{-}}$ for all $x\in ]a-\e, a]$ and 
  $\red{\td E}(x) \subset U_{i_{+}}$ for all $x\in [a, a+\e[$, hence  $\red{\td E}(a)\subset U_{i_{-}}\cap U_{i_{+}}$; and since   $\red{\td E}(a)$ is nonempty then necessarily   ${i_{-}}={i_{+}}=i$ and $\red{\td E}(x)\subset U_{i}$ for all  $x\in ]a-\e, a+\e[$.

This finishes the proof of   Claim \ref{claim}.\end{proof} 

 Then applying Claim \ref{claim}     to $J=I$ and $\xi$ such that $F^{(\xi)}=\0$ we get that ${\red E}\subset U_i$ with necessarily $i=0$ 
 since $U_0\cap {\red E}\not=\0$, which proves that the set ${\red E}$ is connected.
 \end{proof}

 \begin{proposition}\label{H=1} If  $H$ is  an hypergraph with    graph part $G$ then 
 $\dim (H)=\dim (\adh G)=1$.
 \end{proposition} 
 
  \begin{proof}     Since $H$ is connected then   $\dim (H)\geq 1$ and we only have to prove that 
  $\dim (\adh G)\leq 1$.  So fix $(a,b)\in \adh G$; given any neighbourhood $V$    of $(a,b)$ we can find two open intervals $I_a=]a',a''[$ and $J_b=]b',b''[$  such that $(a,b)\in  W=(I_a\times  J_b)\cap \adh G\subset V$.
And we can impose on  the endpoints   $a'$ and $a''$ to be in $P=I\setminus Q$  so that $A=(\{a', a''\}\times J_b)\cap \adh G= \{\bigl(a',f(a')\bigr),\bigl( a'', f(a'')\bigr)\}$
is finite. Also since the function $f$ is nowhere constant on $P$  the set {$B=(I\times\{b',b''\})\cap  \adh G$} is a closed subset of  $\adh G$  of empty interior  in $\R\times\{ b', b''\}$  hence
 $\dim B=0$.
And since  the  boundary  of $W$ in  $\adh G$ is contained in $A\cup B$ then $\dim (\adh G)\leq 1$.
 \end{proof}

\begin{theorem} \label{E=0} If $H$  is  an hypergraph and  $E \subset H$ is such that  $\dim(\pi(E))=0$ then $\dim(E)=0$.  
\end{theorem} 
 
  \begin{proof}  We shall first   prove the theorem  under the additional assumption  that   $E\cap G$ is dense in $E$.

Fix   $(a,b)\in E$ and two intervals
  $I_a=]a-\e,a+\e[$ , $J_b=]b-\e,b+\e[$ for some
  $\e>0$. We have to find      a clopen neighbourhood $W$ of $(a,b)$   in $E$  such that  $W\subset  I_a\times  J_b$.   We distinguish the two cases:
\smallskip

-- If $a\in P$ then $(a,b)\in G$ and $b=f(a)$.    Then by the continuity of $f$  we can find two open intervals  
$I'$ , $J'$ such that   $(a,b)\in  I'\times  J' $,    $\adh{ I'\times  J' }\subset I_a\times  J_b$  and  $f(P\cap I')\subset J'$.
 Since $\pi(E)$ is of empty interior  we can also impose  on the endpoints of  $I'$ to be in $I'\setminus \pi(E)$ so that $I'\cap \pi(E)$ is clopen in $\pi(E)$. Then    $W= E\cap \pi^{-1}(I')$ 
 is a clopen subset of $E$ and  by the choice of $I'$ we have $W\cap G \subset  I'\times  J' $, hence by the density of $G\cap E$ in $E$ we have 
 $W \subset \adh{W\cap G}\subset \adh{ I'\times  J' }\subset I_a\times  J_b$.
\smallskip
 
-- If $a\in Q$ then applying  clause b)  of Lemma \ref{hyper} we can fix   two  continuous functions   $g: I\setminus \{a\}\to\R $   and  $h:  P\cup \{a\}\to\R $ 
  such that $f(x)=g(x)+h(x)$ for all $x\in  P$. Since $H(a)$ is of empty interior we can find  
  $b',b''\in \R \setminus  H(a)$ with  $b'<b<b''$ and $b''-b'< \dfrac \e 2$,  
  and   an  open interval $I'$ containing $a$ on which the oscillation of $h$ is $< \dfrac \e 2$; since $\pi(E)$ is zero-dimensional  then  it is of empty interior and we can impose on 
the endpoints of $I'$ to be in $I \setminus \pi(E)$.  Set $Z=I'\cap \pi(E)\setminus   \{a\}$; since $g$ is continuous  on $Z$ the three sets
\begin{align*}
 U&=\{x\in Z: \  b'-h(a)<g(x)<b''-h(a)\} \\
U_{-} &=\{x\in Z: \ \ g(x) < b'-h(a) +\abs{x-a}\} \\
U_{+}&=\{x\in Z: \  \ g(x) > b''-h(a) -\abs{x-a}\}
\end{align*}
 constitute an open covering of $Z$ which is zero-dimensional, hence we can find in $Z$
 a  clopen covering $(U', U'_-, U'_+)$ finer than $(U, U_-, U_+)$.

 We claim that  each of the following three sets:
 \begin{align*}
W&=\bigl(\{a\}\times ]b' ,b''[ \,\cup \,\pi^{-1}(U') \bigr)\cap E \\
  W_- &=\bigl( \{a\}\times ]-\infty,b'[ \,\cup \,\pi^{-1}(U'_-)  \bigr)\cap E \\
W_+&= \bigl( \{a\}\times ]b'', +\infty,[\, \cup\, \pi^{-1}(U'_+)  \bigr) \cap E
\end{align*}
 is open, hence clopen,  in  $E$.  We detail the proof  for  $W$,  the argument being  similar for the other two sets. 
Since $\pi^{-1}(U')\cap E$ is open in  $E$ we only need to show that any element of the form $(a,y)\in E$ with
$ y\in]b' ,b''[ $ is in the interior of $W$. So suppose by contradiction  that $(a,y)=\lim_j (x_j,y_j)$ with 
$(x_j,y_j)\in W_-\cup W_+$. Then necessarily $x_j\not=a$ for co-finitely many $j$'s; so we may suppose that
for all $j$, $x_j\in U_-\cup U_+$.  
Since $E\cap G$ is dense in $E$  we can also assume that $x_j\in P$ hence $y_j=f(x_j)$. 
Then 
$$\lim_jg(x_j)= \lim_jf(x_j)- \lim_jh(x_j)=y-h(a)\in ]b'-h(a) ,b''-h(a)[$$
hence for large $j$ we have: 
$$ b'-h(a)+\abs{x_j-a} \leq g(x_j)\leq b''-h(a) - \abs{x_j-a} $$
 which is a contradiction since $x_j  \in U_-\cup U_+$. 

Hence $W$ is a clopen subset of $E$ and clearly $(a,b)\in W$. To finish the proof we have to show  that 
$W\subset  I_a\times  J_b$. Again by the density of $G\cap E$ in $E$ it is sufficient to show 
$W\cap G\subset  I_a\times  J_b$. So let $(x,y)\in W\cap G$ then $y=f(x)=g(x)+h(x)$ 
and since the oscillation of  $h$ on $I'$ is $<\dfrac \e 2$ then
$$b'-\dfrac \e 2 \leq b'-h(a)+h(x) \leq y\leq b''-h(a)+ h(x)\leq b''+\dfrac \e 2  $$
and since  $b''-b'< \dfrac \e 2$  then $\abs{y-b}<\e$ so $(x,y)\in I_a\times  J_b$. 

This finishes the proof under the additional density condition and we go back now  to the general case.
Starting from an arbitrary set $E$ such that  $\dim(\pi(E))=0$ consider the closed subset 
$E'={\adh{E\cap G}}^E=\adh{E\cap G}\cap E$ of $E$.
Then ${E'\cap G}={E\cap G}$ and  ${\adh{E'\cap G}}^E = {\adh{E\cap G}}^E=E'$  so  $E'$ satisfies the density condition. Hence  $\dim(E')=0$ 
and since  $E= E'\cup\bigcup_{a\in Q}E\cap( \{a\}\times \R)$ and each $E\cap( \{a\}\times \R)$ is a  zero-dimensional closed subset of $E$, then $\dim(E)=0$.
  \end{proof}

 \red{ \begin{corollary} \label{dimE} Let $H$ be  an hypergraph.  For any  set  $E\subset H$    the following are equivalent:
 \begin{listeroman}
 \item $\dim(E)=0$.
 \item $\dim(\pi(E))=0$.
\end{listeroman}
\end{corollary} 
 \begin{proof} Suppose that $\dim(E)=0$. If $\pi(E)$ were not zero-dimensional   then $\pi(E)$ would contain a non trivial interval $J$  hence by Theorem \ref{Hconnex} the set $E'=E\cap J\times \R$ would be a non trivial connected subset of $E$, a contradiction.
The converse implication follows from Theorem \ref{E=0}.
  \end{proof}

 \begin{corollary} \label{dimF}   For a  closed subset  $F$  of     an hypergraph  $H$ the following are equivalent:
 \begin{listeroman}
 \item $F$ is of empty interior  in $H$.
 \item $\dim(F)=0$.
 \item $\dim(\adh{\pi(F)})=0$.
\end{listeroman}
\end{corollary}

 \begin{proof}   $(i) \Rightarrow (ii)$: Suppose that $F$ is of empty interior  in $H$.   
  If $\pi(F)$ were not zero-dimensional   then $\pi(F)$ would contain a nonempty  open  interval $ J$. Then for all $x\in J\cap P$, $(x,f(x))\in F$ hence  $ F\supset \adh{ \{(x,f(x)) :x\in J\cap P\}}$ and it follows that $F$  contains the  nonempty open subset $ H\cap (J\times\R)$ of $H$, a contradiction.

\sk

 $(ii) \Rightarrow (iii)$:  Observe first that by compactness $\adh{\pi(F)}= \pi( \adh F)$. Also since the grah part $G$ is  a closed subset of  $P\times\R $   then necessarily  $\adh F \setminus F \subset Q \times \R  $   so $\pi(\adh F)\setminus\pi(F)\subset Q$.

 Suppose now that dim($F$)=0.
If    $\pi(\adh F)$ were not   zero-dimensional then it  would be of nonempty  interior in $\R $ hence would contain some open interval  $I_0$  and then   $I_0\setminus Q\subset \pi(F)$, hence  $F$ would   contain the closure in $H$ of the graph of  $f_{\vert I_0\cap P}$ which 
is an hypergraph too, and this is impossible since $F$ is zero-dimensional. 

\sk

 $(iii) \Rightarrow (i)$: Suppose that $\dim(\adh{\pi(F)})=0$.   Then by Corollary \ref{dimE} $\dim(F)=0$ and if $F$ were  not of empty interior then by Remark \ref{rem} b)  $F$ would contain  an hypergraph $H'$, and   then  we would have that $\dim(H')=0$ which contradicts Theorem ~\ref{Hconnex}.   \end{proof}  
}  
  
  \begin{corollary} \label{F0(H)}  
If $H$ is a   $\Gd$   hypergraph then  $\FF_0(H)$ is  a $\Gd$ subset of $\FF(H)$.
  \end{corollary}

 \begin{proof} 
  Set $\hat R=[-\infty, +\infty]$ and  fix   a countable basis $(U_n)_n$ of open sets in $\hat R$.
 Observe that for  any set $A\subset \R$ the set
 $\adh A^\R$ is of empty interior in $\R$ if and only if the set $\adh A^{\hat\R}$ is of empty interior in $\hat\R$.
 
Embed $H$ in the compact space $\hat R\times \hat R$ and consider on $\FF(H)$ the corresponding admissible topology $\hat\tau_0$.
Then by  Corollary~\ref{dimF}  for $F\in\FF(H)$ we have:
\begin{align*}
F\in\FF_0(H) &\iff \adh {\pi(F)}^{\R} \text{ is of empty interior in } {\R}  \\
 &\iff \adh {\pi(F)}^{\hat\R} \text{ is of empty interior in } {\hat\R}  \\
&\iff \forall n, \ \exists m , \ \adh{U_m}^{\hat\R}\subset U_n,  \  \pi(F) \cap \adh{U_m}^{\hat\R} =\0\\
 &\iff \forall n, \ \exists m , \ \adh{U_m}^{\hat\R}\subset U_n,  \  F \cap \hat\pi^{-1}(\adh{U_m}^{\hat\R})=\0
\end{align*}
where $\hat\pi$ denotes the projection mapping
from $\hat R\times \hat R$ onto the first factor. 
Since $\hat\pi^{-1}(\adh{U_m}^{\hat\R})$ is compact then
by \ref{F(X)} b) condition  ``$F \cap \hat\pi^{-1}(\adh{U_m}^{\hat\R})=\0$" on $F$ is open,   hence 
$\FF_0(H)$ is  a  $\borm 2$ subset of $(\FF(H),\hat\tau_0$).

Now if $\hat \tau$ is any   admissible topology on $\FF(H)$ then $\hat \tau$ is first Baire class isomorphic to $\hat \tau_0$, so      $\FF_0(H)$  is a     $\borm 3$ subset of $(\FF(H),\hat\tau$),  hence  by Corollary \ref{KLWI} $\FF_0(H)$  is actually a     $\borm 2$ subset of $(\FF(H),\hat\tau$).  
\end{proof}  
 
  \begin{corollary}\label{K=0} Any compact subset    of    an hypergraph    is of empty interior hence
 zero-dimensional. 
  \end{corollary}

 \begin{proof}  Let  $K$ be a compact subset of $H$ and suppose that $U\times V$ is an open subset of $\R^2$ such that  $\0\not=(U\times V)\cap H\subset K$.   Pick any element $a \in U  \cap Q$; since $\adh G(a)$ is a non trivial interval we can also   pick some element   $b\in V\cap  \adh G(a)\setminus H(a)$, so $(a,b)  \in  
 (\adh G \setminus H) \cap (U\times V)$.  We can then find $(a_j,b_j)\in G\cap   (U \times V) $ such that  $(a,b)=\lim_j(a_j,b_j)$;   and since $G$ is a graph then    $H(a_j)=K(a_j)=\{b_j\}$ hence   $(a_j,b_j)\in K$  and so    $(a,b)\in K\subset H$, which is a contradiction.   
 
 Hence $K$  is  closed and of empty interior and the last part of the conclusion  follows then from
Corollary~\ref{dimF}.
 \end{proof}  
 
 \begin{corollary} 
An    hypergraph  is a connected space which contains no \hbox{non trivial} continuum.
 \end{corollary}

Observe that an   hypergraph $H$ cannot be decomposable  into  a countable union  of  closed  zero-dimensional or compact subsets  since by Corollary \ref{K=0} 
it would  be a countable union  of  closed  zero-dimensional subsets hence by Theorem \ref{dim}  $H$  itself would be zero-dimensional which contradicts Proposition \ref{H=1}.


\end{section}

\begin{section}{The maximum descriptive class of $\FF_0(X)$ when $X$ is Polish}  \label{maxpol}

We can now prove the following complement to Corollary \ref{F0=GSxi=3-complete}

\begin{theorem} \label{F0=GSxi=2-complete} For any $\Gd$ hypergraph  $H$ there exists a $\borm 2$ subset $X$ of\/  $2^\wo\times H$ such that   $\FF_0(X)$ is      $\Game \bora  2$-hard. \end{theorem}

\begin{proof}
 We check that $H$ and its graph part $G$ satisfy all the assumptions of
 Theorem \ref{F0=GSxi-complete}  for $\xi=2$: By Proposition \ref{H=1} $\dim(H)=1$ and it is easily seen as in Remark \ref{rem} c) that if $W$ is a nonempty open subset of $H$ and if a $\s$-compact set $T= \bigcup_n F_n$ (with $T_n$ compact) contains $ G\cap W$ then some $F_n$ contains the whole interval $\adh G(x)$ for some $x\in Q\cap W$ hence $T$ meets $\R^2\mns H$. Thus $G\cap W$ is not relatively $\s$-compact in $H$.  
 Finally observe that if $D\subset G$ then $\pi(D)\cap \pi(H \mns D)=\0$ hence if 
 $D$ is  dense in $G$ then $\pi(D)$ is   dense in  $\pi(H)$  hence  $\pi(H \mns D)$ is of empty interior in $\pi(H)$, so $\dim(\pi(H \mns D))=0$ and   by Theorem \ref{E=0}   $\dim ( H\mns D)=0$. 
 \end{proof} 
 

\begin{ssection}{\bf Back to games:}
{\rm  
For the next result we need to consider games in which one of the Players, namely Player II, is required to make his moves in $\{0,1\}$. Thus if   $\g=(n_0, n_1,n_2 \cdots)\in \wo^\wo$ is any infinite run in such a  game then
 $n_{2k+1}\in\{0,1\}$. The main advantage of this variation is that in this case  the space $\Sigma_{\rm II}^*$ of all strategies for  Player II  is  now compact: $\Sigma_{\rm II}^*=\{0,1\}^{(\wo^{<\wo}\setminus \{\0\})}\approx 2^\wo$ , 
while for Player I the situation is unchanged:   $\Sigma_{\rm I}^*=\wo^{(2^{<\wo})}\approx \wo^\wo$. 
The basic observation is that any $\bora 2$ game on $\wo$  can be reduced (in a precise sense) to a $\bora 2$  game of the previous form. To state this properly  let us introduce a notation. Let $\star$   denote the operation on $\wo^\wo$ defined by:
$$\a\star \b(n)=\left\{ 
\begin{array}{ll}
\a(k) &\text{ if } n=2k\\
\b(k) &\text{ if } n=2k+1\end{array}
\right.$$
Then the mapping $(\a,\b)\mapsto \a\star \b$ is clearly a homeomorphism from $\wo^\wo\times\wo^\wo$ onto $\wo^\wo$.
Notice that if we view $\a$ and $\b$ as two (trivial) strategies  for Player I and Player II respectively  then $\a\star \b$ is just the infinite run obtained by the $\star$ operation introduced in \ref{GG}.a.  

Set:
\begin{align*}
\wo^\wo\star2^\wo&= \{\a\star \b:  \  (\a, \b)\in \wo^\wo\times 2^\wo\}\\
&= \{\g\in \wo^\wo: \ \forall k, \ \g(2k+1)\in \{0,1\} \,  \}
\end{align*}


We need  the following variation of Lemma \ref{Univ}. 
It is worth noting that the situation is not symmetrical and one cannot replace in this statement    $\bora 2$  by 
$\borm 2$, equivalently Player I by Player II,  while   any Borel game  of rank $\xi \geq 3$ on $\wo$ can be reduced to a game   where both Players' moves are  in $\{0,1\}$.

}\end{ssection}

 \begin{lemma} \label{Univ*} There exists a $\bora 2$ set 
 $V\subset 2^\wo\times(\wo^\wo\star2^\wo)$   such that the set $\Game V$ is $\Game\bora 2$-complete. \end{lemma}
 
 \begin{proof} Given any $\e\in 2^{<\wo}\cup2^\wo$ let  
 $\nu=\Psi(\e) \in \wo^{<\wo}\cup \wo^\wo $ be  the  increasing enumeration of the set $N=\{j<\abs \e: \ \e(j)=1\}$    and let $\b=\Phi( \e)\in \wo^{\abs \nu}$ be  defined   by: $\b(0)=\nu(0)$  \text{ and } $ \b(k+1)=\nu({k+1})-\nu(k)-1.$ 
 Observe that  $\abs {\Psi(\e)}=\abs {\Phi(\e)} \leq \abs \e$.

The   mapping $\Phi$  thus defined is clearly   one-to-one on  the set $P\cup Q$ where  
$P=\{\e \in  2^\wo: \ \forall j, \ \exists k> j, \ \e(k)=1\} $ and $Q=\{\e \in  2^{<\wo}: \  \e(\abs {\e}-1)=1\}$ with   
$\Phi(P)=\wo^\wo$ and $\Phi(Q)= \wo^{<\wo}$; and  we shall denote by  $\Phi_*:2^{<\wo}\cup2^\wo\to  Q\cup P$ its inverse.  
Moreover the  mapping  $\Phi_{\vert P}:P\to  \wo^\wo$  is a homeomorphism.
\sk

Then to any   set $A\subset \wo^\wo$ we associate  a set 
$A^*\subset \wo^\wo\star2^\wo$ defined as follows:  an element   $\a\star \e\in\wo^\wo\star2^\wo$ belongs to $A^*$ if:

  -- either  $\abs {\Psi(\e)}=\abs {\Phi(\e)}<\infty$   
  
  -- or   $\abs {\Psi(\e)}=\abs {\Phi(\e)}=\infty$ and if $\nu=\Psi(\e)$ and 
  $\b=\Phi(\e)$   then $(\a \circ \nu)\star \b \in A$

\begin{claim} \label{A*}    Player I wins the game $\bG_A$ if and only if he wins the game $\bG_{A^*}$.
\end{claim}
\begin{proof}

-- If Player I has a winning strategy $\s$ in  $\bG_A$  he constructs   a  strategy $\s^*$  in
  $\bG_{A^*}$ by playing as follows:  If in an infinite run in  $\bG_{A^*}$   in which Player II plays some $\e\in 2^\wo$, then following $\s^*$ Player I plays $\a\in \wo^\wo$ defined inductively by:
 $\a(0)=\s(\0)$  and for $k>0$:
$$\a(k)=\left\{ 
\begin{array}{ll} \s(\Phi(\e_{\vert k})) & \text{if }  \e(k-1)=1  \\
0 & \text{if not}
\end{array}
\right.
$$
If  $\e\not\in P$ then by   definition of $A^*$ Player I wins the run. 
Otherwise $\e\in P$ and if $\nu=\Psi(\e)$ and $\b=\Phi(\e)$ then from the definition of $\s^*$ we have
$\a \circ \nu=\s( \b)$ and since $\s$ is winning in $\bG_A$  then $(\a \circ \nu)\star \b \in A$  so by definition of $A^*$,  $ \a\star \e \in A^*$, hence   $\s^*$ is winning in $\bG_{A^*}$.

 -- Conversely if Player I has a winning strategy $\s^*$ in  $\bG_{A^*}$    he constructs a   strategy  $\s$ in   
 $\bG_A$  by   playing as follows:   In an infinite run in   $\bG_A$ in which Player II plays some $\b\in \wo^\wo$ then following $\s$  Player I plays $\a\in \wo^\wo$ defined  for all    $k$ by: $\a(k)= \s^*(\Phi_*(\b_{\vert j_k}))$ where $(j_k)=\Psi (\Phi_*(\b))$ is the enumeration of the set $\{j: \e(j)=1\}$ for $\e=\Phi_*(\b)$. Notice that 
 $\b_{\vert j_k}\in Q$ hence $\s$ is well defined. Moreover if  $\a'= \s^*(\e)$ then 
$\a=\a'\circ \nu$ and since $\s^*$ is winning  in  $\bG_{A^*}$ then by definition of $A^*$ we have $ \a\star \b \in A$, hence   $\s$ is winning in $\bG_{A}$.
 \end{proof}
 
 \sk
 
Let  $U\subset 2^\wo\times \wo^\wo$  and 
  $V\subset 2^\wo\times (\wo^\wo\star2^\wo)$ defined by  $V(\d)=(U(\d))^*$ for all $\d\in  2^\wo$.

 \begin{claim} \label{UV} 
 If $U$ is  a  $\bora 2$ set then $V$ is  a   $\bora 2$  set  too.
\end{claim}
\begin{proof} Observe that for any set $A\subset \wo^\wo\star2^\wo$ and any $\gamma=(\a\star\e)\in\wo^\wo\star2^\wo$ we have
$$\g\not \in A^*\iff \e\in P \ \wedge \ \a\star \Phi(\e)\not\in A$$
hence for any $\d\in  2^\wo$:
$$(\d,\g)\not \in V\iff \e\in P \ \wedge \ (\d,\a\star \Phi(\e))\not\in U$$
So letting   $\pi: 2^\wo\times \wo^\wo\to \wo^\wo$ and $\rho:2^\wo\times P \to 2^\wo\times \wo^\wo$ denote  the mappings defined 
by: $$\pi (\d,\g)=\e \quad \text{ and } \quad \rho(\d,\g)=(\d,\a\star \Phi(\e))$$  
we have  
 $$2^\wo\times (\wo^\wo\star2^\wo)\setminus V= \pi^{-1}(P) \cap \rho^{-1}( 2^\wo\times\wo^\wo\setminus U) $$
and since  the mappings $\pi$ and $\rho$ are continuous and $P$ is a $\Gd$ set it follows that $2^\wo\times (\wo^\wo\star2^\wo)\setminus V$ is
a $\Gd$ subset, hence $V$ is a $\bora 2$ subset, of    $2^\wo\times (\wo^\wo\star2^\wo)$.
 \end{proof}
  
 To  finish  the proof of Lemma \ref{Univ*}  let  $U\subset 2^\wo\times \wo^\wo$ be a $\bora 2$-complete  set   given by Lemma \ref{Univ}  and consider the set $V\subset 2^\wo\times (\wo^\wo\star2^\wo)$  as above. By 
 Claim \ref{UV} $V$ is a $\bora 2$ set hence $\Game V$ is $\Game\bora 2$ set and we now check that the set $\Game V$ is $\Game\bora 2$-hard.
 
 So let $A\subset  2^\wo$  be any  $\Game \bora 2$ set. By 
Lemma \ref{Univ*} there exists a continuous mapping $\ph: 2^\wo\to 2^\wo$ which reduces the set $A$ to $\Game U$, hence by Claim \ref{A*} we have:

\begin{align*}
\e\in A&\iff \ph(\e)\in \Game U\iff \text{Player I wins } {\bG}_{(\Game U)(\ph(\e))}\\
&\iff \text{Player I wins } {\bG}_{\bigl((\Game U)(\ph(\e))\bigr)^*}={\bG}_{V(\ph(\e))}
\iff \ph(\e)\in \Game V
\end{align*}
which proves that $\ph$   reduces the set $A$ to $\Game V.$ 
  \end{proof}

We recall that it is not known  whether  the set  $\FF_0(\I^2\setminus \Q^2)$ is $\ca$.

\begin{theorem} \label{I^5/Q^5} 
$\FF_0(\I^5\setminus \Q^5)$ is     \hbox{$\Game \bora 2$-hard}. 
\end{theorem}

\begin{proof}  We shall actually construct a space $Y$ for which the set $\FF_0(Y)$ is 
$\Game \bora  2$-hard and which can be embedded as a closed subset in $\I^5\setminus \Q^5$. It will follow  that $\FF_0(Y)$ is a closed subset of  $\FF_0(\I^5\setminus \Q^5)$ hence $\FF_0(\I^5\setminus \Q^5)$ is $\Game \bora  2$-hard too.

The definition of the space $Y$    follows    the same general scheme of definition as for the spaces  $X_U$  constructed in the proof of Theorem \ref{F0=GSxi-complete} and will be of the form:
$$Y=\{(\a,\tau, z)\in 2^\wo\times\Sigma_{\rm II}^*\times H: \  
\forall n, \ z\in F_n \rightarrow  (\a, h_n(z)\star\tau)  \not \in V\}$$ 
where   $V$ is the  set given by Lemma \ref{Univ*},  $H$ is a  specific hypergraph that we will construct, 
the set $\Sigma_{\rm II}^*\approx \wo^\wo$  of strategies replaces the previous set $\Sigma_{\rm II}$, 
and  for all $n$, $h_n:F_n\to \Sigma_{\rm I}^*$  is a homeomorphism \red{defined on  a {\it closed} subset $F_n$  of $H$ which will replace the $\borm 2$ set $G_n$ of   the proof of Theorem \ref{F0=GSxi-complete}.} 
  
 Since the $\bora 2$ set $V$ satisfies also Lemma \ref{Univ}  we can apply  Claim \ref{lemxi} to conclude as in Theorem \ref{F0=GSxi-complete} that  $\FF_0(Y)$ is $\Game \bora  2$-hard. Moreover since the $F_n$'s are closed subsets of $H$ the immediate analysis of the previous definition shows that $X$ is indeed a $\borm 2$ set.

Our plan is to show that for a suitable choice of $H$ and the $F_n$'s the space $Y$ above can  be embedded  as a closed subset  in $\I^5\setminus \Q^5$. We point out  that unlike in the proof of Theorem \ref{F0=GSxi-complete} where the $\red G_n$'s were constructed for a given space $H$ in the present case we will have to construct $H$ and the $F_n's$ simultaneously, which  {will} be done in a series of steps.

\medskip

We recall that   the notation $\adh A$ for a subset of $H$ refers to the closure in $\R^2$. 

\begin{lemma} \label{Fn} In any hypergraph $H$ with  graph part $G$
there exists a closed subset   $F$ of $G$  satisfying:
 $F\approx \wo^\wo$, $\adh F \approx  2^\wo$, and     for all $q\in Q$ , card $\bigl (\adh F(q)\bigr)\leq 1$.
\end{lemma}

\begin{proof}
Observe first that given any $q\in Q$ then by  Definition \ref{defhyper}  we can write $f=g+h$ with
 $h:P\cup\{q\} \to \R $ continuous   and 
 $g:I\setminus\{q\} \to \R $    continuous and such that $\Omega g(q)=\adh G(q)$ is a non trivial interval.
Then for any cluster value $\ell \in \Omega g(q)$ and  any neighbourhood  $J$ of $q$  one can find an infinite sequence  $(J_n)_n$ of pairwise disjoint closed non trivial intervals in $J\setminus\{q\}$ with endpoints in $P$ and of arbitrary small diameter such that  $\lim_n d(q, {I_n})=0$ and the function $g_{\vert  \bigcup_n I_n}$ has limit $\ell$ at $q$; and since $h$ is continuous on $P\cup\{q\}$ then $f_{\vert P\cap \bigcup_n I_n}$ has a limit (namely $\ell+h(q)$) at $q$.

Fix an enumeration $(q_n)_{n\in\wo}$ of $Q$; we shall   construct a family 
$(r_s,J_s)_{s\in \wo^{<\wo}}$    satisfying for all $s$:

\begin{listeroman}
 \item  $J_s$  a closed non trivial  interval  with endpoints in $P$, 
  \item $\diam (J_s)\leq 2^{-\abs s}$,
  \item  $J_{s\cat n}\subset J_s$ and $J_{s\cat m}\cap J_{s\cat n}=\0$ if $m\not=n$,
\item  $r_s\in J_s\cap Q$,
\item  if $\abs s= k$ and $q_k\in J_s$ then $r_s=q_k$, 
  \item  $r_s\not\in J_{s\cat n}$ for all $n$ and $\lim_n d(r_s, J_{s\cat n})=0$,
  \item  $f_{\vert P\cap \bigcup_nJ_{s\cat n}}$  has a limit at $r_s$.
\end{listeroman}
  
 Take for $J_\0$ any  open interval   with endpoints in $P$. Suppose that the family 
$( J_s)_{s\in \wo^k}$ is constructed we shall define for each ${s\in \wo^k}$, the element $r_s$ and the family
 $(J_{s\cat n})_{n\in \wo}$. 

First if $q_k\in J_s$ then we define $r_s=q_k$; if not we choose  $r_s$ arbitrarily in  $J_s\cap Q$ so  
conditions {\it (iv)} and {\it (v)} are satisfied. Once $q=r_s\in J_s$ is chosen, observe that since $J_s$ has endpoints in $P$ then $q$ is necessarily in the interior of $J_s$, and  following the previous observations, we can find a family $(J_{s\cat n})_{n\in \wo}$ satisfying the remaining conditions.  Let
 $$F^*= \bigcap_k\bigcup_{s\in \wo^k}  {J_s} \quad \text{and} \quad K= \adh {F^*}$$

\nd 

\begin{claim}: For any $z\in K$ there exists a unique maximal sequence  $\s=(s_j)$ in  $\wo^{<\wo}$ such that:  $z\in  J_{s_j} $,  $\abs {s_j}=j$   and $s_j\subset s_{j+1}$.
 Then:

 -- if $\s$ is finite with last element $s$ then $z\in Q$  and $z=r_s$
 
 -- if $\s$ is infinite  then  $z\in P$ and $\{z\}=\bigcap_j  J_{s_j}$
\end{claim}

\begin{proof} The construction  is by induction on $j$.
 For $j=0$ then necessarily $s_0=\0$   which  is legal since $z\in  \adh {F^*}\subset J_\0$.
 Suppose that $s_j$ is already defined then again since 
 $z\in  K\subset \{r_{s_j}\}\cup \bigcup_n J_{s_j\cat n}$ then:
 
 --  either $z=r_{s_j}$  and then $\abs {s_j}=j$ and $z=q_j$  and in this case 
 $\s= (s_i)_{i\leq j}$ is maximal since $z\not\in J_s$ for any $s\supset s_j$ with $\abs {s}=j+1$
 
 -- or   $z\in J_{s_j\cat n}$ for some unique $n$ and so necessarily $s_{j+1}=s_j\cat n$. 
 
\nd This finishes the construction of $\s$ and:

--  either $\s$ is finite and  $z=r_s$ where $s$ is  last element  of $\s$

--  or $\s$ is infinite and $\{z\}=\bigcap_j  J_{s_j}$ and we have to show that  $z\in P$. If not
then we would have $z=q_k$ for some $k$  hence $q_k\in J_{s_k}$  with $r_{s_k}\not = q_k$ 
which is in contradiction with condition {\it (v)}. 
\end{proof}

This defines a mapping  $h: K\to \wo^{<\wo}\cup \wo^\wo \approx \mathcal P(\wo)\approx 2^\wo$
which is clearly a homeomorphism  such that 
$F^*= h^{-1}(\wo^{\wo})=K\cap P\approx \wo^{\wo}$ 
    and    $K\setminus F^*= h^{-1}(\wo^{<\wo})= K\cap Q=\{r_s: \ s\in\wo^{<\wo}\}$;
hence $F^*$ is a closed subset of $P$. Finally    set $F=\pi^{-1}(F^*)$:   since $\pi$ is a homeomorphism from $G$ onto $P$ then  $\wo^\wo \approx F^*\approx F$ and $F$ is a closed subset of $G$; and since  by condition {\it (vii)} $f_{\vert F^* }$  admits a continuous extension   $g: K\to \R$   then  $\adh F=\Gr(g)$ hence for all 
 $q\in Q$, $\adh F (q)$ is either  $\{g(q)\}$ if $q=r_s\in K$, or else empty.
 \end{proof}

\begin{lemma} \label{Fnn} There exists an hypergraph $H$ with  graph part $G$  and a sequence $(F_n)$ of  subsets of $H$ satisfying: 

 a) $\adh H\setminus H$ is countable,

b) any nonempty open subset of $H$ contains some $F_n$.

 c) for all $n$, 
  $F_n\approx \wo^\wo$ ,  $\adh F_n \approx  2^\wo$ and $F_n=\adh F_n \cap H\subset G$ 

\end{lemma}

\begin{proof} 
Start from an   hypergraph $H'$ with  graph part $G$  which is co-countable in $\adh G$ (see Remark \ref{rem}.c)  and let  $(W_n)_n$ be a basis of the topology of $H'$ with  for all $n$, $W_n=H'\cap (I_n\times J_n)\not=\0$ for some intervals $I_n, J_n$. Then applying Lemma \ref{Fn}  to  each hypergraph $H_n=H'\cap W_n$ with graph part $G_n=G\cap W_n$, we can find  a  closed subset  $F_n$  of $ G_n\cap W_n$ such that for all $q\in Q$, $\adh F_n (q)$ is   contained in some singleton   $\{q^{(n)}\}$. 
Let $$H=H'\setminus\{(q,q^{(n)}): \ q\in Q \text{ and } n\in \wo\}$$ 
then $\adh G\setminus H=(\adh G\setminus H')\cup(H'\setminus H)  $ is countable hence  $\adh H=\adh G$ and
$\adh H\setminus H$  is countable. In particular for all $q\in Q$,   $\adh G(q)\setminus H(q)$ is countable, so  
$H(q)$ is dense  in $\adh G(q)$, and since $H(q)\subset H'(q)$  then $H(q)$ is   of empty interior in   $\adh G(q)$. Hence $H$ is   an hypergraph   with same graph part $G$. 
Finally  for all $n\in \wo$ and all $q\in Q$, $\bigl(\{q\}\times\adh F_n (q)\bigr)\cap H=\0$,  hence 
$\adh F_n  \cap H\subset G$ and since $F_n$ is a closed subset of $G$ then 
$F_n=\adh F_n  \cap  G=\adh F_n\cap H$.    
\end{proof}

\begin{lemma} \label{Fnnn}
Let   $H$ and $(F_n)$ be as in  Lemma \ref{Fnn},  let $V$ be an arbitrary $\bora 2$ subset  of $\wo^\wo$ and    for all $n$,   let $\Phi_n:2^\wo\times2^\wo\times F_n \to   \wo^\wo$ be a    continuous mapping. Then the space:
$$X=\{(\a,\b, z)\in 2^\wo\times2^\wo\times H: \  
\forall n, \ z\in F_n \rightarrow \Phi_n(\a,\b,z)  \not \in V\}$$ 
 can  be embedded  as a closed subset  in $\I^5\setminus \Q^5$.
\end{lemma}

 \begin{proof} Since $V$ is a $\bora 2$ set and the $F_n$'s are closed subsets of $H$  then $X$ is clearly a $\borm 2$ subset of the   compact space $K= 2^\wo\times 2^\wo\times \adh H$.  
Observe that  if $(\a,\b, z)\in K\setminus X$ then either $z\in D=\adh H \setminus H$ which is countable, or
$z\in F_n\subset \adh F_n \approx 2^\wo $ for some $n$. Hence
$$K\setminus X\subset \bigcup_{z\in D,\, n\in \wo} 2^\wo\times 2^\wo\times (\{z\}\cup  \adh F_n)$$
Since  for all $z\in D$ and all $n$, the set 
$2^\wo\times 2^\wo\times (\{z\}\cup   \adh F_n)$  is a zero-dimensional compact  space  then
 by     Theorem       \ref{dim} $\dim(K\setminus X)=0$, and  we can write  the $\s$-compact space $K\setminus X$ as the union of a countable family $(K_j)_j$ of pairwise disjoint compact sets: 
 $K\setminus X =\bigcup_jK_j$ with  $\lim_j \text{diam} (K_j) =0$.                                                                                                                                                                                               Hence  if $\Delta$ denotes the diagonal of $K$ then $E=\Delta\cup \bigcup_j K_j\times K_j$ is a closed   equivalence relation on $K$. Let $L$ be the quotient (compact)  space $K/E$ and   $\rho:K\to L$  be  the canonical quotient mapping. Since  the restriction of $\rho$ to $X$ is a one-to-one perfect mapping on its image $Y=\rho(X)$ then it is a  homeomorphism from $X$ onto   $Y$; moreover  the  set $L\setminus Y$   is countable as the collection of the $E$-equivalence classes $K_j$. 

Note that by  Proposition \ref{H=1},  $\dim (Y)=\dim (X) \leq \dim (K)=\dim ( \adh G)=1$, and since 
$L\setminus Y$  is   countable          then by Theorem \ref{dim},  $\dim (L) \leq 2$, hence the compact space $L$ can be embedded in $\R^5$.

Fix now a homeomorphism $h_0:L \to L_0\subset \R^5$; then   $D_0=(\Q^5 \setminus L_0)\cup 
h_0 (L\setminus Y)$ is   a countable dense subset of $\R^5$, and since $\R^5$ is a CDH space (\cite{bnt}) then there exists a 
 homeomorphism $h_1:\R^5 \to \R^5$  such that $h_1(D_0)=\Q^5$ and we can 
 assume that the compact space 
$L_1=h_1(L_0)$ is a subset  of $\I^5$. Hence 
 $h=h_1\circ h_0\circ \rho$ induces a  homeomorphism from $X$ onto $L_1 \setminus \Q^5=L_1\cap(\I^5 \setminus \Q^5)$   {which is} a closed subset of  $\I^5 \setminus \Q^5 $. 
 \end{proof}

  {
 \nd{\it \underbar{End of the proof of Theorem \ref{I^5/Q^5}}:}
Starting from $H$ and $(F_n)$  as in  Lemma \ref{Fnn} we fix for all $n$, an arbitrary homeomorphism $h_n:F_n\to \Sigma_{\rm I}^*\approx \wo^\wo$. It is clear then that the space $Y$ defined at the beginning of the proof is of the form considered in Lemma \ref{Fnnn}. Hence $Y$ can be embedded  as a closed subset  in
 $\I^5\setminus \Q^5$ and since $\FF_0(Y)$ is  $\Game\bora 2$-hard then   $\FF_0({\I^5\setminus\Q^5})$ is $\Game\bora 2$-hard too. 
}
   \end{proof}

 \end{section}


 \bibliographystyle{amsplain}

   \end{document}